\documentclass[%
  a4paper,
  onecolumn,
  colorlinks,
]{preprint}

\usepackage[english]{babel}
\usepackage[T1]{fontenc}

\usepackage{amsmath}
\usepackage{amssymb}
\usepackage{amsthm}
\usepackage{mathtools}
\usepackage{stmaryrd}
\SetSymbolFont{stmry}{bold}{U}{stmry}{m}{n}
\usepackage{mathrsfs}

\usepackage{graphicx}
\usepackage{xcolor}
\usepackage{subcaption}
\usepackage{tikz}
\usepackage{pgfplots}
\usepackage{pgfplotstable}
\usetikzlibrary{external}

\usepackage[linesnumbered, lined, algoruled]{algorithm2e}

\usepackage{todonotes}
\makeatletter
\DeclareOldFontCommand{\rm}{\normalfont\rmfamily}{\mathrm}
\makeatother

\pgfplotsset{compat = 1.18}

\usepgfplotslibrary{external}

\tikzset{external/system call = {%
    pdflatex \tikzexternalcheckshellescape
    -halt-on-error
    -interaction=batchmode
    -jobname "\image" "\texsource"}}
\tikzexternaldisable

\theoremstyle{plain}

\theoremstyle{definition}

\newtheorem{definition}{Definition}
\newtheorem{proposition}{Proposition}
\newcommand{\R}{\ensuremath{\mathbb{R}}}

\newcommand{\norm}[1]{\left\|#1\right\|}
\newcommand{\range}{\operatorname{range}}
\newcommand{\Kry}{\mathcal{K}}          

\newcommand{\bA}{\ensuremath{\boldsymbol{A}}}

\newcommand{\bD}{\ensuremath{\boldsymbol{D}}}
\newcommand{\bE}{\ensuremath{\boldsymbol{E}}}
\newcommand{\bF}{\ensuremath{\boldsymbol{F}}}
\newcommand{\bG}{\ensuremath{\boldsymbol{G}}}

\newcommand{\bI}{\ensuremath{\boldsymbol{I}}}

\newcommand{\bL}{\ensuremath{\boldsymbol{L}}}

\newcommand{\bP}{\ensuremath{\boldsymbol{P}}}
\newcommand{\bQ}{\ensuremath{\boldsymbol{Q}}}
\newcommand{\bR}{\ensuremath{\boldsymbol{R}}}
\newcommand{\bS}{\ensuremath{\boldsymbol{S}}}

\newcommand{\bV}{\ensuremath{\boldsymbol{V}}}
\newcommand{\bW}{\ensuremath{\boldsymbol{W}}}

\newcommand{\bb}{\mathbf{b}}
\newcommand{\br}{\mathbf{r}}
\newcommand{\by}{\mathbf{y}}
\newcommand{\bx}{\mathbf{x}}

\newcommand{\bq}{\mathbf{q}}
\newcommand{\bd}{\mathbf{d}}

\newcommand{\bp}{\mathbf{p}}
\newcommand{\bv}{\mathbf{v}}
\newcommand{\bt}{\mathbf{t}}

\newcommand{\be}{\mathbf{e}}
\newcommand{\bz}{\mathbf{z}}

\newcommand{\vnew}{\bv_{\mathrm{new}}}

\DeclareMathOperator*{\argmin}{argmin}

\definecolor{matlabblue}{HTML}{0072BD}
\definecolor{matlabgreen}{HTML}{77AC30}
\definecolor{matlaborange}{HTML}{D95319}

\tikzset{Kron_Style_Bar/.style={
    matlabblue, solid, line width=2pt, mark=none, fill=matlabblue!100}}
\tikzset{KRP_Style_Bar/.style={
    matlaborange, solid, line width=2pt, mark=none,
    mark options={solid}, fill=matlaborange!100}}
\tikzset{Naive_Style_Bar/.style={
    matlabgreen, solid, line width=2pt, mark=none,
    mark options={solid}, fill=matlabgreen!100}}
\tikzset{Naive_Style_line/.style={matlabgreen, solid, line width=2pt}}
\tikzset{Kron_Style_line/.style={matlabblue,  solid, line width=2pt}}
\tikzset{KRP_Style_line/.style={matlaborange, solid, line width=2pt}}

\begin{document}

\title{A Sketched Generalized Krylov Subspace Method
       for Large-Scale Regularization}
\author[1]{\mbox{Davide Palitta}}
\author[2]{and \mbox{Mirjeta Pasha}}

\affil[1]{Dipartimento di Matematica and (AM)$^2$,
Alma Mater Studiorum Universit\`a di Bologna

 Piazza di Porta San Donato  5, I-40127 Bologna, Italy.\authorcr
\email{davide.palitta@unibo.it}, \orcid{0000-0002-6987-4430}}

\affil[2]{Department of Mathematics, Virginia Tech, Blacksburg, VA 24061, USA.\authorcr
       \email{mpasha@vt.edu}, \orcid{0000-0003-4249-2421}}

\shorttitle{Sketching for GKS}
\shortauthor{D. Palitta and M. Pasha}
\shortinstitute{}



\abstract{
The generalized Krylov subspace (GKS) method is an effective projection
technique for large-scale Tikhonov regularization with a general
regularization matrix.
As the subspace expands, however, two computational bottlenecks limit
scalability: the thin QR factorizations of the tall projected matrices
formed by the forward operator and the regularization matrix applied to
the basis, and the full reorthogonalization of each new basis vector
against all previous columns.

We propose a sketched variant, named sGKS, that addresses both
bottlenecks.
The QR factorizations are performed on compressed matrices of much
smaller row dimension, maintained incrementally via rank-one updates.
Moreover, we observe that explicit reorthogonalization can be skipped
entirely without compromising the quality of the approximation
subspace, since no step of GKS relies intrinsically on the
orthogonality of the basis.
The resulting algorithm is independent of the choice of sketching
operator and preserves the approximation quality of the original
method: we show that, in the absence of sketching in the projected
solve, sGKS produces iterates identical to those of standard GKS,
and that the sketched projected solve delivers quasi-optimal residual
norms controlled by the embedding quality.
For more challenging problems where the loss of basis orthogonality
becomes significant, we show that incorporating a small number of
iterative refinement steps in the projected solve restores the
spectral properties of the basis and recovers the full accuracy of
the unsketched method.
Numerical experiments on image deblurring, X-ray computerized
tomography, seismic travel-time tomography, and dynamic computerized
tomography demonstrate that sGKS matches the reconstruction quality
of standard GKS while significantly reducing per-iteration costs and
overall wall-clock time.
}

\maketitle

\section{Introduction}\label{sec:intro}

We consider large-scale ill-posed least-squares problems of the form  \begin{equation}\label{eq:ls}              \min_{\bx \in \R^n} \norm{\bA\bx - \bb}^2_2, \end{equation}  where $\bA \in \R^{m \times n}$ is severely ill-conditioned, and $\bb \in \R^m$ is a noise-contaminated observation                  $\bb = \bb_{\mathrm{true}} + \be$ with $\norm{\be}_2$ small.                     These types of problems arise, e.g., from the discretization of ill-posed operator  equations~\cite{engl1996regularization,Kirsch2011}, including Fredholm integral      equations of the first kind~\cite{Hansen2010,Groetsch1984}, and tomographic imaging~\cite{Natterer2001,NattererWuebbeling2001}. 

It is well-known that, due to the ill-conditioning of $\bA$, minimizing~\eqref{eq:ls} directly yields a solution dominated by amplified noise; regularization is required to stabilize the computation. Many diverse strategies have been proposed during the years. A non-exhaustive list includes truncated singular value decomposition (TSVD)~\cite{Hansen1987}, iterative regularization methods such as CGLS and GMRES~\cite{Hansen1998,Calvetti2002}, hybrid projection methods that combine iterative projection with Tikhonov regularization~\cite{Chung2008,Gazzola2015,ChungGazzola2024}, $\ell_p$--$\ell_q$ regularization~\cite{Huang2017,Lanza2015}, total variation regularization~\cite{RudinOsherFatemi1992}, and learned or data-driven approaches~\cite{Arridge2019, chung2021efficient}. One of the most widely used regularization strategies is due to Tikhonov~\cite{Tikh63}, which replaces~\eqref{eq:ls} by the penalized problem \begin{equation}\label{eq:tik}
    \min_{\bx \in \R^n} \left\{ \norm{\bA\bx - \bb}_2^2 + \lambda\norm{\bL\bx}_2^2 \right\}, \end{equation} where $\bL \in \R^{p\times n}$ is a regularization matrix that encodes prior information about the expected solution, and $\lambda > 0$ is the regularization parameter.  
When $\bL$ is the identity operator (standard-form Tikhonov), efficient methods based on the partial
Lanczos bidiagonalization of $\bA$ are well established
\cite{PaigeSaunders82, Reichel03, CalvettiReichel03}.
For general $\bL$, however, a transformation to standard form via, e.g., the
$\bA$-weighted pseudoinverse $\bL^\dagger_A$~\cite{Elden84} is often impractical;
for instance, when $\bL$ is a discrete derivative operator on a two-dimensional
domain it takes the form of a sum of Kronecker products whose pseudoinverse is
dense and thus expensive to operate with in the large-scale setting.

In \cite{LRV12} the authors proposed the
generalized Krylov subspace (GKS) method that handles general $\bL$ directly
by expanding, at each iteration $k$, a search space $\mathcal V_k \subseteq \R^n$ using the gradient of the Tikhonov
functional.
This gradient explicitly depends on the regularization parameter $\lambda$ which can be updated at each iteration.
While effective, the GKS method inherits several computational bottlenecks
that become prohibitive as the problem and subspace dimensions grow.
In particular, if $\bV_k$ denotes a matrix whose columns represent an orthonormal basis of the current
$\mathcal V_k$, then the full orthogonalization of the newly acquired gradient with respect to $\bV_k$ is performed. Similarly, the
QR factorizations of the projected matrices $\bA\bV_k$ and $\bL\bV_k$ needed for the solution of the projected counterpart of~\eqref{eq:tik} have to be updated.

All these operations scale unfavorably with the
iteration count.
These costs motivate the use of randomized dimensionality reduction
techniques -- sketching -- to accelerate the most expensive operations
while preserving the approximation quality of the subspace.

The practical benefit of sketching is that inner products between
high-dimensional vectors in a subspace $\mathcal{V}\subset\R^n$ can be
replaced by inner products between their sketched counterparts of
dimension $s \ll n$, yielding substantial savings provided that the
sketching operator can be applied efficiently.
For instance, a Gaussian random matrix $\bS \in \R^{s \times n}$ can be used as sketching but requires $O(sn)$ operations per application.
More structured alternatives, such as the subsampled randomized
Hadamard transform (SRHT) or subsampled randomized trigonometric
transforms (SRTTs), reduce this cost to $O(n \log s)$ per
application~\cite{WOOLFEetal2008}; we refer to~\cite{Woodruff14} for a comprehensive
survey of sketching matrices and their computational trade-offs.
Since in our setting the subspace $\mathcal{V}_k$ is a subspace
that is built incrementally as the iteration progresses, the sketching
operator must be an \emph{oblivious} embedding that is drawn
independently of $\mathcal{V}_k$ and requires no a priori knowledge of
the subspace beyond an upper bound on its dimension.

In the recent literature, sketching has been used to reduce the cost of orthogonalization in Krylov methods for, e.g., linear systems~\cite{NakatsukasaTropp2024,BalabanovGrigori22}, matrix function evaluations~\cite{GuettelSchweitzer2022,Palittaetal2025}, and matrix and tensor equations~\cite{Palittaetal2025_bis,Buccietal2025}.
In the context of ill-posed problems, these techniques have been recently adopted also
 for Golub-Kahan-like methods by Chung and
Gazzola~\cite{ChungGazzola25}.

In this paper we employ sketching to enhance 
GKS applied to large-scale regularization problems.
Indeed, both the QR factorizations of $\bA\bV_k$ and $\bL\bV_k$, and the
reorthogonalization inner products $\bV_k^\top \br_k$ need to be computed at each iteration, making the GKS method a natural target
for the randomization strategies described above.
The present work improves the efficiency of the original GKS scheme by (i) noticing that working with an orthonormal basis is not of paramount importance and forcing orthogonality can thus be avoided, and (ii) applying randomized sketching to~\eqref{eq:tik},
extending the ideas of~\cite{ChungGazzola25} from the Golub-Kahan
bidiagonalization setting to the generalized Krylov subspace framework with a
general regularization matrix $\bL \in \R^{p \times n}$.

Here is a synopsis of the paper. In section~\ref{sec:background} we recall some background material. More precisely, we revisit the GKS method and its main computational challenges (section~\ref{sec:GKS}) and the concept of sketching operators (section~\ref{The sketching operator}). Section~\ref{sec:sketching_GKS} sees the main contribution of this work, namely the design of the sketched GKS method (sGKS). Some comments on this novel scheme are provided in section~\ref{sec:analysis}. A panel of diverse numerical results is illustrated in section~\ref{sec:numerics} whereas we draw some conclusions in section~\ref{sec:conclusions}.

\section{Background material}\label{sec:background}

\subsection{GKS}\label{sec:GKS}

The goal of the Generalized Krylov Subspace (GKS) method~\cite{LRV12} is to solve~\eqref{eq:tik} by projecting it onto a sequence of nested subspaces. In particular, GKS
builds a subspace
$\mathcal{V}_k = \range(\bV_k)$, where $\bV_k \in \R^{n \times d_k}$ with
$d_k = \ell + k $ has orthonormal columns, and solves a projected Tikhonov problem at each step. In this paper we assume that a suitable value of the regularization parameter
$\lambda > 0$ is available; its selection is discussed
in~\cite{LRV12} and is orthogonal to the algorithmic
contributions of this paper.

Given an initial subspace $\mathcal{V}_0$, usually selected as the Krylov subspace $\mathcal{V}_0=\Kry_\ell(\bA^\top\bA, \bA^\top\bb)$ for a small $\ell>0$, we look for an approximation to the exact solution $\bx$ in~\eqref{eq:ls} of the form $\bx_k= \bV_{k}\by_k$. The vector $\by_k\in\mathbb{R}^{d_k}$ is the solution of the projected variant of~\eqref{eq:tik}, i.e., 
\begin{equation}\label{eq:projected_withQR}
  \by_k =\arg\min_{\by\in\R^{d_k}}
  \left\|
  \begin{bmatrix}\bA\\
  \sqrt{\lambda}\bL\end{bmatrix}\bV_{k}\by
  -\begin{bmatrix}\bb\\\mathbf{0}\end{bmatrix}
  \right\|_2^2= \arg\min_{\by\in\R^{d_k}}
  \left\|
  \begin{bmatrix}\bR_A^{(k)}\\\sqrt{\lambda}\bR_L^{(k)}\end{bmatrix}\by
  -\begin{bmatrix}(\bQ_A^{(k)})^\top\bb\\\mathbf{0}\end{bmatrix}
  \right\|_2^2,
\end{equation}
where $\bA\bV_{k} = \bQ_A^{(k)}\bR_A^{(k)}$ and $\bL\bV_{k} = \bQ_L^{(k)}\bR_L^{(k)}$ are skinny QR factorizations.

The GKS method expands $\mathcal{V}_k=\text{range}(\bV_{k})$ at each iteration by appending the
(normalized, reorthogonalized) gradient of the Tikhonov functional, i.e.
\begin{equation}\label{eq: gradient}
  \br_{k+1} = \bA^\top(\bA\bV_{k}\by_k - \bb)
         + \lambda\bL^\top(\bL\bV_{k}\by_k).
\end{equation}
In exact arithmetic, $\br_{k+1}$ is orthogonal to $\mathcal{V}_k$. Indeed, a direct computation shows that $\bV_k^\top\br_{k+1}=0$ as this equation corresponds to the normal equations related to~\eqref{eq:projected_withQR}. However, in finite precision arithmetic such orthogonality may fall short to hold and is thus often explicitly imposed. 
In particular, the new column $\vnew$ of $\bV_{{k+1}}=[\bV_{k},\vnew]$
is obtained by orthonormalizing $\br_{k+1}$ against
the current $\bV_{k}$.
Since $\bV_k \in \R^{n \times d_k}$ has orthonormal columns,
$\bV_k^\top \bV_k = \bI_{d_k}$, the orthogonal projector onto
$\range(\bV_k)^\perp$ is
$\bP_k = \bI_n - \bV_k\bV_k^\top \in \R^{n \times n}$.
In exact arithmetic, $\bP_k^2 = \bP_k$ and $\bP_k^\top = \bP_k$.
A single application to the gradient $\br_{k+1}$ suffices to remove all components
in $\range(\bV_k)$:
\begin{equation}\label{eq:single_mgs}
  \bP_k\br_{k+1} = \br_{k+1} - \bV_k(\bV_k^\top\br_{k+1}).
\end{equation}
The subspace expansion then appends
$\vnew = \bP_k\br_{k+1} / \|\bP_k\br_{k+1}\|_2$,
which is exactly orthogonal to every column of $\bV_k$.

In floating-point arithmetic the computed basis satisfies
$\bV_k^\top\bV_k = \bI_{d_k} + \bE$, where
$\|\bE\| = \mathcal{O}(\kappa(\bV_k)\,\varepsilon_\mathrm{mach})$
and $\kappa(\bV_k)$ is the condition number of $\bV_k$~\cite{GiraudLangouRozloznik05}.
Round-off in the inner products accumulates over iterations, so
$\kappa(\bV_k)$ grows with $d_k$.
After one projection pass the residual orthogonality error satisfies
\begin{equation}\label{eq:single_residual}
  \bigl\|\bV_k^\top(\bP_k\br_{k+1})\bigr\|_2
  = \mathcal{O}\!\bigl(\kappa(\bV_k)\,\varepsilon_\mathrm{mach}\,\|\br_{k+1}\|_2\bigr),
\end{equation}
which can become significant in later iterations.
Applying the projection a second time reduces this to
$O(\kappa(\bV_k)^2\varepsilon_\mathrm{mach}^2\|\br_{k+1}\|_2)$---effectively
zero---while leaving the mathematical solution unchanged since $\bP_k^2 = \bP_k$
in exact arithmetic. The double application is therefore written as:
\[
   \mathbf{\widetilde v}_{\rm new} = (\bI_n - \bV_k\bV_k^\top)^2\br_{k+1},\qquad \vnew=\mathbf{\widetilde v}_{\rm new}/\|\mathbf{\widetilde v}_{\rm new}\|,
\]
so that $\bV_{k+1}=[\bV_k,\mathbf{v}_{\rm new}]$.
A common practical refinement avoids the second pass when unnecessary.
After the first projection, one can compute the loss ratio
\begin{equation}\label{eq:loss_criterion}
  \ell_{\mathrm{loss}} = \frac{\|\bV_k^\top(\bP_k\br_{k+1})\|_2}{\|\bP_k\br_{k+1}\|_2}.
\end{equation}
If $\ell_{\mathrm{loss}} < \tau$ for a threshold $\tau$ (e.g., $10^{-10}$),
the second pass is skipped; only when $\ell_{\mathrm{loss}} \geq \tau$ is the
correction applied.


Once the space is expanded, the solution of~\eqref{eq:projected_withQR} requires the computation of the skinny QR factorizations of 
$\bA\bV_{k+1}=[\bA\bV_{k},\bA\vnew] \in \R^{m \times d_{k+1}}$ and
$\bL\bV_{k+1}=[\bL\bV_{k},\bL\vnew] \in \R^{p \times d_{k+1}}$.
Rather than recomputing their thin QR factorizations from scratch, one can
\emph{update} the existing factorizations of $\bA\bV_{k}$ and $\bL\bV_{k}$ by appending one column---a rank-1
extension.
In particular, if $\bA\bV_{k+1} = \bQ_A^{(k+1)}\bR_A^{(k+1)}$ with $\bQ_A^{(k+1)}\in\R^{m\times d_{k+1}}$,
$\bR_A^{(k+1)}\in\R^{d_{k+1}\times d_{k+1}}$ are such that 
\begin{equation}\label{eq:updqr}
  \bQ_A^{(k+1)} = \bigl[\bQ_A^{(k)},\;\bq_\mathrm{new}\bigr],
  \qquad
  \bR_A^{(k+1)} = \begin{bmatrix}\bR_A^{(k)} & \bp_A \\ \mathbf{0}^\top & \rho_A\end{bmatrix}
,
\end{equation}
we compute the vectors $\bq_\mathrm{new}$, $\bp_A$, and the scalar $\rho_A$ as follows
\begin{align*}
  \bp_A &= (\bQ_A^{(k)})^\top \bA\vnew \in \R^{d_k}
  &&\text{(project onto existing } \bQ_A^{(k)}\text{)}
 \\
  \br^\perp_A &= \bA\vnew - \bQ_A^{(k)}\bp_A \in \R^m
  &&\text{(orthogonal residual)}
 \\
  \rho_A &= \|\br^\perp_A\|_2
  &&\text{(new diagonal entry)}\\
  \bq_\mathrm{new} &= \br^\perp_A / \rho_A \in \R^m
  &&\text{(new } \bQ_A^{(k+1)}\text{ column).}\\
\end{align*}
Once these quantities have been computed, we can cheaply update the right-hand side in~\eqref{eq:tik}. Indeed, if 
$\bd_A^{(k)}:=(\bQ_A^{(k)})^\top \bb$, then $(\bQ_A^{(k+1)})^\top\bb$ gains one new scalar entry:
\begin{equation}\label{eq:dA_update}
  \bd_A^{(k+1)} = \begin{bmatrix}\bd_A^{(k)} \\ \bq_\mathrm{new}^\top\bb\end{bmatrix}
  \in\R^{d_{k+1}},
\end{equation}
which costs $O(m)$ (a single dot product with the fixed vector $\bb$).
The identical update applies to $\bL\bV_{k+1} = \bQ_L^{(k+1)}\bR_L^{(k+1)}$.

In addition to having a maximum number of iteration $\mathtt{maxit}$, a common stopping criterion for GKS checks the distance between two consecutive approximations, namely $\|\bx_k-\bx_{k-1}\|_2$, and stops GKS as soon as this quantity becomes smaller than a specified threshold {\ttfamily tol}. Since $\bx_k=\bV_k\by_k$, and $\bV_k$ has orthonormal columns, it holds
\begin{equation}\label{eq:stoppingcriterion_cheap}
\|\bx_k-\bx_{k-1}\|_2=\|\bV_k(\by_k-[\by_{k-1};0])\|_2=\|\by_k-[\by_{k-1};0]\|_2,    
\end{equation}
which means that the stopping criterion can be implemented so that to manipulate only small dimensional objects.

The overall GKS process is summarized in Algorithm~\ref{alg:gks}.

\begin{algorithm}[t]
\DontPrintSemicolon\SetAlgoLined
\caption{GKS Tikhonov regularization \cite{LRV12}}\label{alg:gks}
\KwIn{$\bA \in \R^{m\times n}$, $\bb \in \R^m$, $\bL \in \R^{p\times n}$,
      initial basis $\bV_0$ with $\bV_0^\top\bV_0 = \bI$, regularization parameter $\lambda$, thresholds $\tau,\mathtt{tol}>0$, maximum number of iterations $\mathtt{maxit}$}

\For{$k = 0, 1, 2, \ldots, \mathtt{maxit}$}{
  Compute/update QR factorizations $\bA\bV_{k} = \bQ_A^{(k)}\bR_A^{(k)}$ and $\bL\bV_{k} = \bQ_L^{(k)}\bR_L^{(k)}$\;
  Solve~\eqref{eq:projected_withQR} to get $\by_{k}$\;
  \If{$k>0$ \rm{and} $\|\by_k-[\by_{k-1};0]\|_2\leq\mathtt{tol}$}{
Go to line~\ref{alg:finalline}\;
  }
  Compute the residual
    $\br_{k+1} $ as in~\eqref{eq: gradient}\;
Compute 
    $\widetilde{\bv}_{\rm new} = (\bI - \bV_k\bV_k^\top)\br_{k+1}$ and $\ell_{\rm loss}$ as in~\eqref{eq:loss_criterion}\;
    \If{$\ell_{\rm loss} \geq \tau$}{
Compute 
    $\widetilde{\bv}_{\rm new} = (\bI - \bV_k\bV_k^\top)\widetilde{\bv}_{\rm new}$\;
    }
  Normalize
    $\vnew = \widetilde{\bv}_{\rm new}/ \|\widetilde{\bv}_{\rm new}\|_2$ and set
    $\bV_{k+1} = [\bV_k,\;\vnew]$\;
}
\Return $\bx_{k} = \bV_k\by_{k}$\label{alg:finalline}
\end{algorithm}


\subsection{The sketching operator}\label{The sketching operator}
In this section we revise the main features of sketching operators. We point the interested reader to~\cite{Woodruff14} for a more thorough presentation.

\begin{definition}\label{def:sketch_op}
Given a $q$-dimensional subspace $\mathcal{V}\subset\mathbb{R}^n$, a linear map $\bS \in \R^{s \times n}$ is said to be a randomized $(\varepsilon,\delta,q)$-subspace embedding of $\mathcal{V}$
if and only if
\begin{equation}\label{eq:norm_equiv}
  (1 - \varepsilon)\,\|\bx\|_2^2
  \;\le\; \|\bS\bx\|_2^2
  \;\le\; (1 + \varepsilon)\,\|\bx\|_2^2,
  \quad \text{for all }\, \bx \in \mathcal{V},
\end{equation}
with failure probability $\delta$ (see, e.g.,~\cite[Definition 2.3]{BalabanovGrigori22}). 
\end{definition}
The chain of inequalities in~\eqref{eq:norm_equiv} implies that
\begin{equation}\label{eq:eps_embed}
  \bigl|\langle \bx, \by \rangle
  - \langle \bS\bx,\, \bS\by \rangle\bigr|
  \;\le\; \varepsilon\,\|\bx\|_2\,\|\by\|_2,
  \quad \text{for all }\, \bx, \by \in \mathcal{V},
\end{equation}
also holds with high probability. Similarly, we can define the $\bS^\top\bS$-norm as $\|\bx\|_{\bS^\top\bS}^2:=\bx^\top\bS^\top\bS\bx$ on $\mathcal{V}$. Indeed, the symmetric positive semidefinite matrix $\bS^\top\bS$ defines a proper positive definite inner product on $\mathcal{V}$ with probability at least $1-\delta$.

In principle, any (random) matrix satisfying Definition~\ref{def:sketch_op} can be used
for our algorithmic purposes; the method we design and the analysis that follows
are independent of the particular choice of $\bS$. On the other hand, the subspace $\mathcal{V}$ we want to embed is not known a-priori. This motivates the use of a particular class of subspace embedding, the so-called oblivious subspace embeddings, that can be constructed by solely relying on the dimension $q$ of $\mathcal{V}$ (or an upper bound thereof) and not the space itself.
Common choices in this class include Gaussian embeddings,
subsampled randomized Hadamard transforms (SRHT), and subsampled randomized trigonometric transforms (SRTT).  
For each of these types of embeddings, laws relating the sketching dimension $s$, the space dimension $q$, and the quality of the sketching encoded in $\varepsilon$ and $\delta$, can be found in the literature. For instance, if $\bS\in\mathbb{R}^{s\times n}$ is a Gaussian matrix, then choosing $s=\mathcal{O}(\varepsilon^{-2}\log q\log\frac{1}{\delta})$ would ensure~\eqref{eq:eps_embed}; see, e.g.,~\cite[Theorem 2]{Sarlos06}. Similarly, it was shown in \cite{Tropp11}
 that if $s=\mathcal{O}(\varepsilon^{-2}(k+\log\frac{n}{\delta})\log\frac{q}{\delta})$, then a SRTT is an oblivious $\varepsilon$-subspace
embedding for any $q$-dimensional subspace of $\mathbb{R}^n$.
Nonetheless, numerical evidence suggests that selecting the
smaller sketching dimension 
$s=\mathcal{O}(\varepsilon^{-2}\frac{q}{\delta})$ works well in 
practice; see, e.g.,~\cite[Section~9]{HMT11}.

In all our numerical experiments in section~\ref{sec:numerics} we always adopt a SRTT as sketching. In particular, we choose
\begin{equation}\label{eq:sketching_trig}
\bS=  \sqrt{\frac{n}{s}}\bD\bF\bE\in\mathbb{R}^{s\times n},
\end{equation}
where $\bE\in\mathbb{R}^{n\times n}$ is a diagonal matrix with Rademacher entries
(i.e., the diagonal entries are randomly chosen as $\pm 1$ with equal
probability), $\bD\in\mathbb{R}^{s\times n}$
 contains $s$ randomly selected rows of
the identity matrix, and $\bF\in\mathbb{R}^{n\times n}$ is the discrete cosine transform.
The selection of SRTTs is due to their low cost. Indeed, in~\cite{WOOLFEetal2008} it has been shown how
 performing $\bS \bx$ costs only $\mathcal{O}(n\log s)$
 floating point operations (flops) if the action of $\bS$ is
judiciously implemented. On the other hand, we already mention here that, due to our rather basic implementation, the cost of applying the matrix $\bS$ in~\eqref{eq:sketching_trig} to an $n$-dimensional vector is still  $\mathcal O(n \log n)$ flops.

\section{Sketched GKS (sGKS)}\label{sec:sketching_GKS}

In this section we now develop the sketched version of GKS.
The central idea is to reduce the cost of the most expensive steps in GKS (basis orthogonalization and solution of~\eqref{eq:tik})
with cheaper variants, involving
randomized approximations that preserve the the essential properties of the underlying problem while maintaining sufficient accuracy.
To this end, we adopt \emph{two distinct} sketching operators: $\bS_r \in\mathbb{R}^{s_r\times h_r}$ where the role of $r\in\{A,L\}$ will become clear shortly whereas the value of $h_r$ will be dictated by the dimension of the objects we want to sketch.

We start by focusing on the orthogonalization of the basis. As mentioned in section~\ref{sec:GKS}, the basis computed by GKS is orthogonal in exact arithmetic.
However, orthogonality in finite precision arithmetic is enforced by running, e.g., a Gram-Schmidt step. We propose to skip such explicit orthogonalization. Indeed, no step of GKS makes actual use of the orthogonality of the basis. In particular, an explicit update of the QR factorizations involved in~\eqref{eq:tik} is performed. This is in contrast to what is done in, e.g., Krylov methods for linear systems where a (Petrov-) Galerkin condition on the residual is implicitly imposed by formulating the projected problems in terms of quantities stemming from the orthogonalization of the basis; this is the role of the celebrated Arnoldi relation, see, e.g.,~\cite{Saad2003}. This equivalency requires the basis to be orthogonal. Similarly, the residual norm can be often cheaply computed by relying, once again, on the orthogonality of the basis. None of these steps is part of GKS. Moreover, also in the context of the solution of linear systems it has been shown how the orthogonality of the basis is not paramount. The most important aspect is that the adopted approximation space keeps growing as the iterations proceed, namely the newly computed basis vector encodes new directions not belonging to the already constructed subspace; see, e.g.,~\cite{SimonciniSzyld2005}.

 Our sketched GKS scheme still looks for an approximate solution $\bx_k$ to \eqref{eq:ls} of the form $\bx_k=\bW_k\by_k$ where, however, the $d_k$ columns of $\bW_k$ are not orthonormal anymore. In particular, given an initial basis ${\bW}_0\in\mathbb{R}^{n\times \ell}$, which plays the same role as $\bV_0$ in section~\ref{sec:GKS}, we define
 ${\bW}_k=[{\bW}_0,\br_1,\ldots,\br_k]\in\mathbb{R}^{n\times d_k}$. The gradient of the Tikhonov functional $\br_k$ are computed as in~\eqref{eq: gradient} where, however, the vectors $\by_k$ do not solve~\eqref{eq:projected_withQR} anymore but an approximation thereof, as shown below.
 
The second step of our procedure is the solution of the least-squares problem~\eqref{eq:tik}. In our setting it becomes
\begin{equation}\label{eq:tik_new}
  \min_{\bx=\bW_k\by \in \R^n} \left\{ \norm{\bA\bx - \bb}_2^2 + \lambda\norm{\bL\bx}_2^2 \right\}=\min_{\by \in \R^{d_k}} \left\{ \norm{\bA\bW_k\by - \bb}_2^2 + \lambda\norm{\bL\bW_k\by}_2^2 \right\}.
\end{equation}

To reduce its computational cost, we employ a sketch-and-solve approach; see, e.g.,~\cite{Sarlos06, Blendenpik}. In particular, given two sketching matrices $\bS_A\in\mathbb{R}^{s_A\times n}$ and $\bS_L\in\mathbb{R}^{s_L\times p}$ consisting of two subspace embeddings for $\text{range}(\bA\bW_{k})$ and $\text{range}(\bL\bW_{k})$, respectively, we then solve 
\begin{equation}\label{eq:tik_new_sketched}
 \min_{\by \in \R^{d_k}} \left\{ \norm{\bS_A(\bA\bW_k\by - \bb)}_2^2 + \lambda\norm{\bS_L\bL\bW_k\by}_2^2 \right\},
\end{equation}
i.e., 
\begin{align}\label{eq:projected_sketched_bis}
  \by_k =&\arg\min_{\by\in\R^{d_k}}
  \left\|
 \begin{bmatrix}\bS_A\bA{\bW}_{k}\\
  \sqrt{\lambda}\bS_L\bL{\bW}_{k}\end{bmatrix}\by
  -\begin{bmatrix}\bS_A\bb\\\mathbf{0}\end{bmatrix}
  \right\|_2^2 =\arg\min_{\by\in\R^{d_k}}
  \left\|
 \begin{bmatrix}\bR_A^{(k)}\\
  \sqrt{\lambda}\bR_L^{(k)}\end{bmatrix}\by
  -\begin{bmatrix}(\bQ_A^{(k)})^\top\bS_A\bb\\\mathbf{0}\end{bmatrix}
  \right\|_2^2.
\end{align}

In~\eqref{eq:projected_sketched_bis}, we employ the QR factorizations $\bQ_A^{(k)}\bR_A^{(k)}=\bS_A\bA{\bW}_{k}=[\bS_A\bA{\bW}_{k-1},\bS_A\bA\br_{k}]$, $\bQ_L^{(k)}\bR_L^{(k)}=\bS_L\bL{\bW}_{k}=[\bS_L\bL{\bW}_{k-1},\bS_L\bL\br_{k}]$ that can be easily updated as shown in section~\ref{sec:background} for the standard GKS procedure.

In Algorithm~\ref{alg:gks_sk} we report the overall sketched GKS procedure. Notice that, due to the non orthogonality of $\bW_k$, $\|\bx_k-\bx_{k-1}\|_2$ cannot longer be recast in terms of small dimensional objects only. However, computing $\bx_k$ and $\bx_{k-1}$ does not increase the overall computational cost of the methods as these vectors are required to define the gradients $\br_{k}$ and $\br_{k-1}$, respectively.

\begin{algorithm}[t]
\DontPrintSemicolon\SetAlgoLined
\caption{Sketched GKS Tikhonov regularization\label{alg:gks_sk}}
\KwIn{$\bA \in \R^{m\times n}$, $\bb \in \R^m$, $\bL \in \R^{p\times n}$, sketchings $\bS_A$, and $\bS_L$,
      initial basis ${\bW}_0$, regularization parameter $\lambda$, thresholds $\mathtt{tol}>0$, maximum number of iterations $\mathtt{maxit}$}
\For{$k = 0, 1, 2, \ldots, \mathtt{maxit}$}{
  Compute/update QR factorizations $\bS_A\bA\bW_{k} = \bQ_A^{(k)}\bR_A^{(k)}$ and $\bS_L\bL\bW_{k} = \bQ_L^{(k)}\bR_L^{(k)}$\;
  Solve~\eqref{eq:projected_sketched_bis} to get $\by_{k}$\;
  \If{$k>0$ \rm{and} $\|\bW_k(\by_k-[\by_{k-1};0])\|_2\leq\mathtt{tol}$}{
Go to line~\ref{alg:finalline_sk}\;}
Compute the residual
$\br_{k+1} $ as in~\eqref{eq: gradient} and set ${\bW}_{k+1}=[{\bW}_{k},\br_{k+1}]$\;
}
\Return $\bx_{k} = {\bW}_k\by_{k}$\label{alg:finalline_sk}\;
\end{algorithm}

\subsection{Some considerations on sGKS}\label{sec:analysis}

In this section we would like to report some considerations on our novel numerical procedure, illustrating the similarities and differences between sGKS and its standard counterpart GKS.

For the time being, assume that we do not employ $\bS_A$ and $\bS_L$, i.e., the vector $\by_k$ providing our approximate solution $\bx_k=\bW_k\by_k$ is such that
\begin{equation}\label{eq:sGKS_noSASL}
     \by_k=\argmin_{\by \in \R^{d_k}} \left\{ \norm{\bA\bW_k\by - \bb}_2^2 + \lambda\norm{\bL\bW_k\by}_2^2 \right\}.
     \end{equation}
In this case, our sGKS would be equivalent to the standard GKS procedure, provided $\text{range}({\bW}_0)=\text{range}(\bV_0)$. Indeed, in this case, it holds $\text{range}(\bW_k)=\text{range}(\bV_k)$ for any $k>0$ as shown in the following proposition.

\begin{proposition}\label{Prop_equivalent}
Let $\bx_k^{\rm GKS}=\bV_k\by_k^{\rm GKS}$ be the solution obtained by running $k$ iterations of GKS (Algorithm~\ref{alg:gks}). Similarly, let $\bx_k^{\rm sGKS_0}=\bW_k\by_k^{\rm sGKS_0}$ be the solution obtained by running $k$ iterations of sGKS (Algorithm~\ref{alg:gks_sk}), where $\by_k^{\rm sGKS_0}$ satisfies~\eqref{eq:sGKS_noSASL} in place of solving~\eqref{eq:tik_new_sketched}. Moreover, assume that $\text{range}({\bW}_0)=\text{range}(\bV_0)$.
Then, $\bx_k^{\rm GKS}=\bx_k^{\rm sGKS_0}$.
\end{proposition}

\begin{proof}
Since $\text{range}({\bW}_0)=\text{range}(\bV_0)$ by assumption, there exists a nonsingular matrix $\bG_0\in\mathbb{R}^{\ell \times\ell}$ such that $\bW_0=\bV_0\bG_0$. We can thus write 
\begin{align*}
\by_0^{\rm sGKS_0}=&\,\argmin_{\by \in \R^{d_0}} \left\{ \norm{\bA\bW_0\by - \bb}_2^2 + \lambda\norm{\bL\bW_0\by}_2^2 \right\}\\
=&\,\argmin_{\by \in \R^{d_0}} \left\{ \norm{\bA\bV_0\bG_0\by - \bb}_2^2 + \lambda\norm{\bL\bV_0\bG_0\by}_2^2 \right\}.
\end{align*}
By performing the change of variable $\bt:=\bG_0\by$, the problem becomes
$$\min_{\bt \in \R^{d_0}} \left\{ \norm{\bA\bV_0\bt - \bb}_2^2 + \lambda\norm{\bL\bV_0\bt}_2^2 \right\},$$
whose solution is $\by_0^{\rm GKS}$. This means that $\by_0^{\rm GKS}=\bG_0\by_0^{\rm sGKS_0}$. Therefore
$$\bx_0^{\rm sGKS_0}=\bW_0\by_0^{\rm sGKS_0}=\bV_0\bG_0\by_0^{\rm sGKS_0}=\bV_0\by_0^{\rm GKS}=\bx_0^{\rm GKS}.$$
Having the same approximate solution at the first step implies that also the first gradient of the Tikhonov functional computed by the two methods, namely $\br_1$, will be the same. Therefore, $\text{range}(\bW_1)=\text{range}(\bV_1)$ and the proof can be completed by induction.
\end{proof}

The key ingredient in the proof of Proposition~\ref{Prop_equivalent} is the solution of~\eqref{eq:sGKS_noSASL} in place of~\eqref{eq:tik_new_sketched}. As soon as we introduce the sketching $\bS_A$ and $\bS_L$, the equivalency between the two methods -- GKS and sGKS -- no longer holds. 
Indeed, if $\by_k^{\rm sGKS}$ denotes the solution to~\eqref{eq:tik_new_sketched}, we can show that the gradients
$$\br_{k+1}^{\rm sGKS}=\bA^\top(\bA\bW_{k}\by_k^{\rm sGKS} - \bb)
         + \lambda\bL^\top(\bL\bW_{k}\by_k^{\rm sGKS}),$$
and 
$$\br_{k+1}^{\rm sGKS_0}=\bA^\top(\bA\bW_{k}\by_k^{\rm sGKS_0} - \bb)
         + \lambda\bL^\top(\bL\bW_{k}\by_k^{\rm sGKS_0}),$$
 belong to the same space, namely, 
$\text{range}((\bA^\top\bA+\lambda\bL^\top\bL)\bW_{k}         )+\bA^\top \bb$, but we cannot say that 
$\text{range}([\bW_k,\br_{k+1}^{\rm sGKS}])= \text{range}([\bW_k,\br_{k+1}^{\rm sGKS_0}])$, even though this is often the case in practice; see Experiment 4 in section~\ref{sec:numerics}. 
The main goal when constructing 
$\br_{k+1}^{\rm sGKS}$ is that this vector provides new insightful information compared to those already encoded in $\bW_k$.
This would certainly be the case if, e.g., $\bW_k^\top \br_{k+1}^{\rm sGKS}=0$. However, this orthogonality does not hold, in general. Indeed, $\bW_k^\top \br_{k+1}^{\rm sGKS}$ does not amount to the normal equations formulation of~\eqref{eq:tik_new_sketched}. The latter would be 
$$\bW_k^\top\bA^\top\bS_A^\top \bS_A(\bA\bW_k\by-\bb)+\lambda\bW_k^\top\bL^\top \bS_L^\top \bS_L\bW_k\by.$$
This formulation may induce one to enlarge the current subspace with the vector 
$\bz_{k+1}=\bA^\top\bS_A^\top \bS_A(\bA\bW_k\by_k^{\rm sGKS}-\bb)+\lambda\bL^\top \bS_L^\top \bS_L\bW_k\by_k^{\rm sGKS}$ in place of $\br_{k+1}^{\rm sGKS}$. While this selection ensures the orthogonality of $[\bW_k,\bz_{k+1}]$, it also produced poor approximate solutions in our numerical experiments.
Indeed, $\bz_{k+1}$ corresponds to the gradient of the sketched Tikhonov functional, namely using it would be equivalent to applying the standard GKS method to 
$$  \min_{\bx \in \R^n} \left\{ \norm{\bS_A\bA\bx - \bb}_2^2 + \lambda\norm{\bS_L\bL\bx}_2^2 \right\}.$$
We thus recommend the use of $\br_{k+1}^{\rm sGKS}$ as new basis vector, possibly sacrificing the orthogonality of the basis. 

A possible alternative to retrieve the orthogonality of the basis would be to try to reduce the distance between $\by_k^{\rm sGKS}$ and $\by_k^{\rm sGKS_0}$. Indeed, since $\bW_k^\top\br_{k+1}^{\rm sGKS_0}=0$, having a small $\|\by_k^{\rm sGKS}-\by_k^{\rm sGKS_0}\|_2$ could potentially lead to having 
$\bW_k^\top\br_{k+1}^{\rm sGKS}\approx 0$.
 
While it is well-known that sketching techniques for least squares problems produce solutions that are quasi-optimal in terms of the residual norm of the original problem -- see Proposition~\ref{Prop_res_sketching} below -- they often lead to poor solutions in terms of error, namely $\|\by_k^{\rm sGKS}-\by_k^{\rm sGKS_0}\|_2$ may be large. However,
 Algorithm~\ref{alg:gks_sk} can be enhanced with some steps of an iterative refinement scheme similar to those proposed in, e.g.,~\cite{Iterative_refinement_sketching}, to reduce $\|\by_k^{\rm sGKS}-\by_k^{\rm sGKS_0}\|_2$. In section~\ref{sec:numerics} we will illustrate the impact of this approach on both the properties of the constructed basis and the computational performance of the overall solution process.

We conclude this section by showing that 
$\by_k^{\rm sGKS}$ produces quasi-optimal residual norms.

\begin{proposition}\label{Prop_res_sketching}
Let $\bS_A$ be an $\varepsilon_A$-subspace embedding of $\text{range}(\bA\bW_k)$. Similarly, let $\bS_L$ be an $\varepsilon_L$-subspace embedding of $\text{range}(\bL\bW_k)$. Then,
$$\norm{\bA\bW_k\by_k^{\rm sGKS} - \bb}_2^2 + \lambda\norm{\bL\bW_k\by_k^{\rm sGKS}}_2^2\leq \frac{1+\gamma}{1-\theta}\bigg(\norm{\bA\bW_k\by_k^{\rm sGKS_0} - \bb}_2^2 +\lambda\norm{\bL\bW_k\by_k^{\rm sGKS_0}}_2^2\bigg),$$
where  $\theta:=\min\{\varepsilon_A,\varepsilon_L\}$ and $\gamma:=\max\{\varepsilon_A,\varepsilon_L\}$.
\end{proposition}

\begin{proof}
The result simply comes from applying the $\varepsilon$-subspace embedding property~\eqref{eq:norm_equiv} along with the optimality of $\by_k^{\rm sGKS}$ in terms of the sketched problem~\eqref{eq:tik_new_sketched}. In particular, it holds 
{\small
\begin{align*}
\norm{\bA\bW_k\by_k^{\rm sGKS} - \bb}_2^2+\lambda\norm{\bL\bW_k\by_k^{\rm sGKS}}_2^2\leq & \, \frac{1}{1-\varepsilon_A}\norm{\bS_A(\bA\bW_k\by_k^{\rm sGKS} - \bb)}_2^2+\frac{\lambda}{1-\varepsilon_L}\norm{\bS_L\bL\bW_k\by_k^{\rm sGKS}}_2^2\\
\leq& \,
\frac{1}{1-\theta}\bigg(\norm{\bS_A(\bA\bW_k\by_k^{\rm sGKS} - \bb)}_2^2+\lambda\norm{\bS_L\bL\bW_k\by_k^{\rm sGKS}}_2^2
\bigg)\\
\leq& \,
\frac{1}{1-\theta}\bigg(\norm{\bS_A(\bA\bW_k\by_k^{\rm sGKS_0} - \bb)}_2^2+\lambda\norm{\bS_L\bL\bW_k\by_k^{\rm sGKS_0}}_2^2
\bigg)\\
\leq &\,
\frac{1}{1-\theta}\bigg((1+\varepsilon_A)\norm{\bA\bW_k\by_k^{\rm sGKS_0} - \bb}_2^2+\lambda(1+\varepsilon_L)\norm{\bL\bW_k\by_k^{\rm sGKS_0}}_2^2
\bigg)\\
\leq&\, \frac{1+\gamma}{1-\theta}\bigg(\norm{\bA\bW_k\by_k^{\rm sGKS_0} - \bb}_2^2+\lambda\norm{\bL\bW_k\by_k^{\rm sGKS_0}}_2^2
\bigg),
\end{align*}
}
where $\theta=\min\{\varepsilon_A,\varepsilon_L\}$ and $\gamma=\max\{\varepsilon_A,\varepsilon_L\}$.
\end{proof}

\section{Numerical Experiments}\label{sec:numerics}

In this section we illustrate the performance of sGKS by testing it on problems stemming from image deblurring, computerized tomography, travel-time seismic tomography, and a dynamic computerized example. 
All computations were carried out in MATLAB\textsuperscript{\textregistered} R2026a with about 15 significant decimal digits running on a laptop computer with an 
Apple\textsuperscript{\textregistered} Core(TM)i7-8750H CPU @2.20GHz with 128GB of RAM. All experiments are executed in double precision.

For all test cases, we perturb the vector of measurements $\bb_{\mathrm{true}}$ with white Gaussian noise, i.e., the noise vector $\be$ has mean zero and a scaled identity covariance matrix; 
we refer to the ratio 
\begin{equation}
\sigma=\|\be\|_{2}/\|\bA\bx\|_{2}
\end{equation} as the noise level and $\bb=\bb_{\mathrm{true}}+\be$ will be the corrupted right-hand side in~\eqref{eq:ls}.  

In each case we compare the standard GKS algorithm with incremental QR
updates (denoted GKS) against the sketched variant (denoted sGKS).
The dimension of the starting Krylov subspace $\text{range}(\bV_0)=\text{range}(\bW_0)=\Kry_\ell(\bA^\top\bA, \bA^\top\bb)$  is $\ell = 3$. 
 We employ a stopping criterion based on the
relative change
\begin{equation}\label{eq:RC}
  \mathrm{RC}_k
  = \frac{\|\bx_k - \bx_{k-1}\|}{\|\bx_{k-1}\|}
  < 10^{-5},
\end{equation}
for both GKS and sGKS. Unless stated otherwise,
the regularization 
matrix used is the first-order finite-difference operator given by
\begin{equation}\label{eq: reg_matrix}
  \bL = \begin{bmatrix} \bI_{n_y} \otimes \bD_{n_x} \\
        \bD_{n_y} \otimes \bI_{n_x} \end{bmatrix}
  \in \R^{p \times n}, \quad \text{where} \quad
  \bD_{n_d} = \begin{bmatrix} -1 & 1 & & \\ & \ddots & \ddots & \\ & & -1 & 1
          \end{bmatrix} \in \R^{(n_d-1) \times n_d}, 
\end{equation}
with $d$ denoting $x$ and $y$ direction.
For the sGKS method the sketch
dimensions $s_A$ and $s_L$ are set
as small multiples of the maximum subspace dimension.
Reconstruction quality is measured by the relative reconstruction error
\begin{equation}
\mathrm{RRE} = \|\bx_k - \bx_{\mathrm{true}}\|_2 / \|\bx_{\mathrm{true}}\|_2
.
\end{equation}
In all numerical experiments we fix the regularization parameter at a manually tuned optimal value~$\lambda^*$ so that differences between methods reflect algorithmic performance only. Automatic selection strategies such as discrepancy principle, generalized cross validation, and the L-curve criterion~\cite{Reichel03,CalvettiReichel03,LRV12, buccini2021generalized} are compatible with the proposed framework but we do not pursue them here. The interested reader is referred to \cite{buccini2024software} and the references therein for further details.

Test problems are generated with Matlab software; the IRtools~\cite{GazzolaNagyPer19} and
AIR~Tools~II~\cite{hansen2018air} packages. Software in Python for GKS can be found in \cite{pasha2025trips}. The Matlab implementations of the methods described in this paper, along with a selection of the tests illustrated in this section, will be available at \url{https://github.com/mpasha3/sGKS} once the manuscript is accepted for publication.

\subsection{Experiment 1: Image Deblurring}\label{sec:deblur}

In the first numerical experiment we consider image deblurring, a classical inverse problem in which the goal is to
recover a sharp image from an observed image that has been degraded by
blur and contaminated by noise.
More specifically, we consider the recovery of an $n_x \times n_y$ HST (Hubble Space
Telescope) image shown in Figure~\ref{fig:deblur_setup}(a), such that
$\bx_{\mathrm{true}} \in \R^n$ with $n = n_x n_y$, and the observed
(blurred and noisy image) by $\bb \in \R^n$ shown in Figure~\ref{fig:deblur_setup}(c). We set $n_x = n_y = 400$, yielding $n = n_x n_y = 160{,}000$ unknowns, and add
white Gaussian noise at relative level $\sigma = 0.1\%$.
The blurring process is modeled as a convolution with a known
point-spread function (PSF), displayed in Figure~\ref{fig:deblur_setup}(b), which gives rise to a forward operator
$\bA \in \R^{n \times n}$ encoding the discrete convolution and can be generated through \texttt{PRblur} from IRtools~\cite{GazzolaNagyPer19}.
In this experiment we use the regularization operator $\bL$ in~\eqref{eq: reg_matrix} with $p=319,200$ rows.

To investigate the sensitivity of sGKS to the sketch dimension, we run
the method with five values of~$s\equiv s_A=s_L$:
$s = 405$, $1{,}000$, $5{,}000$, $10{,}000$, and $15{,}000$,
alongside standard GKS as the baseline.
All methods are run for $400$ iterations at the optimal regularization
parameter~$\lambda^*$ for each selection of $s$. Reconstructed images by GKS, sGKS with $s=405$ and sGKS with $s= 15,000$ are shown in Figure~\ref{fig:deblur_rec} (a), (b), and (c), respectively. Visually, the reconstructions provided by sGKS are of the same quality as those computed by GKS, even for the smallest sketching dimension $s$.

Further, we examine the convergence behavior of GKS and sGKS under two higher noise levels,
$\sigma = 1\%$ and $\sigma = 5\%$.
Figure~\ref{fig:deblur_rre} reports the relative reconstruction error
as a function of the iteration count for all six methods (GKS and sGKS for five different values of $s$) at each noise
level.
The pentagon markers indicate the iteration at which each method would
have been terminated had we employed the stopping criterion~\eqref{eq:RC}. 

We start by focusing on the case of $\sigma = 1\%$ (Figure~\ref{fig:deblur_rre} top).  First, we observe that all sGKS variants converge to a final RRE that is close to the GKS baseline, and the gap narrows monotonically as~$s$ increases:
the largest sketch ($s = 15{,}000$) achieves an
RRE of $0.1132$, essentially matching GKS ($0.1132$), while the
smallest sketch ($s = 405$) saturates at $0.115$.
Second, the convergence trajectories of sGKS with large~$s$ are
nearly indistinguishable from GKS throughout the entire iteration
history, confirming that a sufficiently large sketch preserves the
spectral approximation quality of the projected problem.
Third, the stopping iteration identified by the relative-change
criterion~\eqref{eq:RC} varies across methods:
GKS stops at iteration~89, while sGKS with
$s = 15{,}000$ stops at iteration~94 at a comparable RRE, and sGKS
with $s = 405$ triggers the criterion much earlier (iteration~27)
due to the stagnation induced by the coarse sketch.

Finally, the same qualitative behavior is observed at the higher
noise level ($\sigma = 5\%$ -- Figure~\ref{fig:deblur_rre} bottom), with all methods converging to a
higher floor determined by the noise, and the ranking among sketch
sizes remains consistent.

\begin{figure}[t!]
\centering
\begin{subfigure}[t]{0.3\textwidth}
\centering
\includegraphics[width=\textwidth]{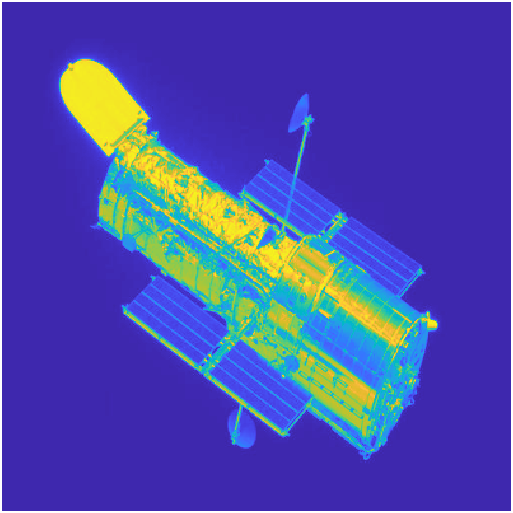}
\caption{}
\label{fig:deblur_true}
\end{subfigure}
\hfill
\begin{subfigure}[t]{0.3\textwidth}
\centering
\includegraphics[width=\textwidth]{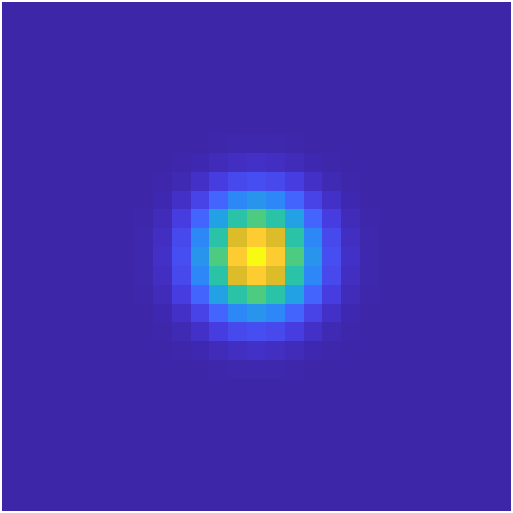}
\caption{}
\label{fig:deblur_psf}
\end{subfigure}
\hfill
\begin{subfigure}[t]{0.3\textwidth}
\centering
\includegraphics[width=\textwidth]{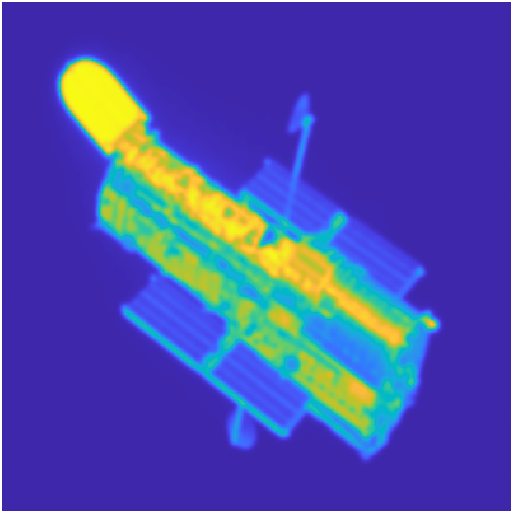}
\caption{}
\label{fig:deblur_blurred}
\end{subfigure}
\caption{Experiment 1 (Image deblurring).
(a)~Ground truth (HST image, $400 \times 400$), (b)~Motion point-spread function at mild blur level, and
(c)~Observed image after convolution and noise contamination with $\sigma = 1\%$.}
\label{fig:deblur_setup}
\end{figure}

\begin{figure}[t!]
\centering
\begin{subfigure}[t]{0.3\textwidth}
\centering
\includegraphics[width=\textwidth]{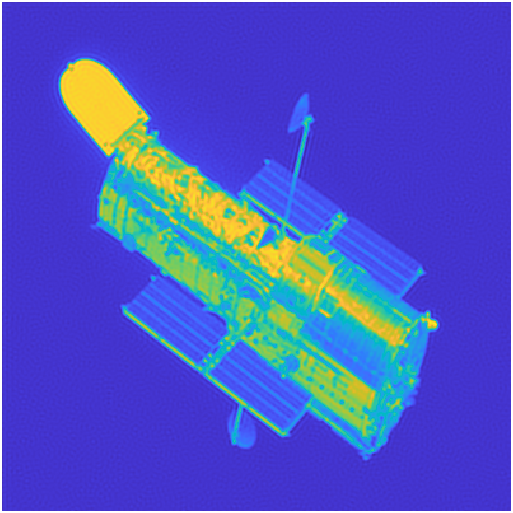}
\caption{}
\label{fig:deblur_true}
\end{subfigure}
\hfill
\begin{subfigure}[t]{0.3\textwidth}
\centering
\includegraphics[width=\textwidth]{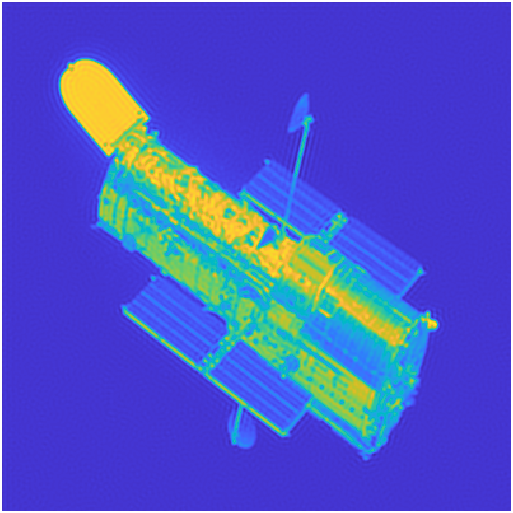}
\caption{}
\label{fig:deblur_psf}
\end{subfigure}
\hfill
\begin{subfigure}[t]{0.3\textwidth}
\centering
\includegraphics[width=\textwidth]{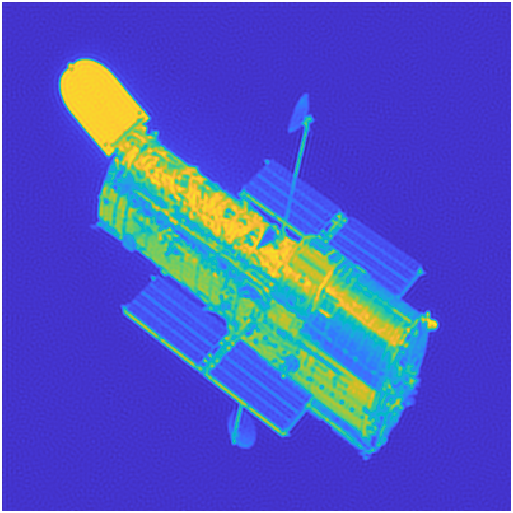}
\caption{}
\label{fig:deblur_blurred}
\end{subfigure}
\caption{Experiment 1 (Image deblurring) ($n_x = n_y = 400$): reconstructions at optimal~$\lambda^*$. (a) ~GKS (no sketching). (b) ~sGKS with $s = 405$ (smallest sketch dimension). (c) ~sGKS with $s = 15{,}000$ (largest sketch dimension).} 
\label{fig:deblur_rec}
\end{figure}

\begin{figure}[t!]                  \centering
\includegraphics[width=\textwidth]{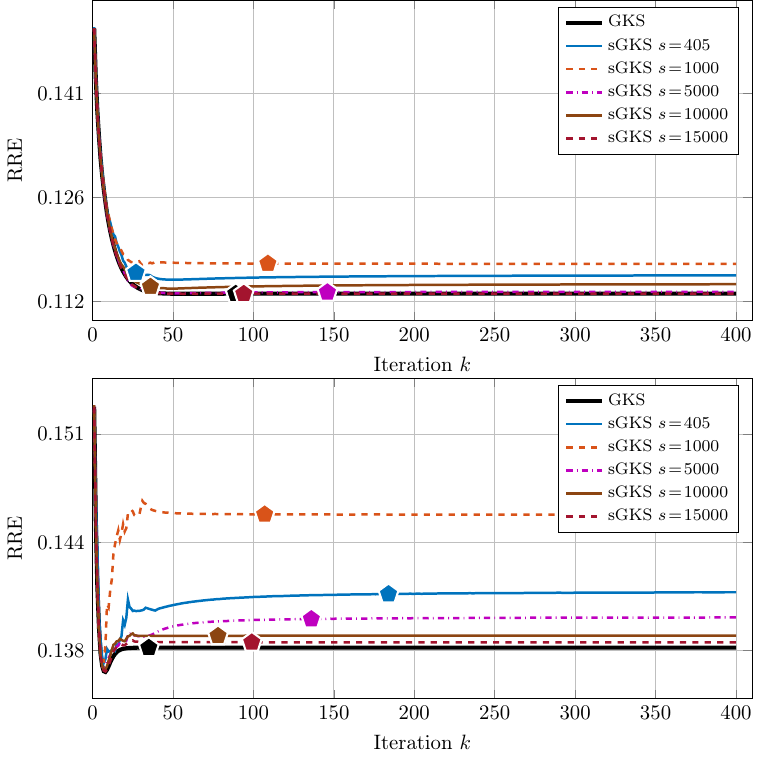}
\caption{Experiment 1 (Image deblurring). Relative reconstruction error versus iteration at the optimal $\lambda^*$ for GKS and sGKS. Top: $\sigma=1\%$. Bottom: $\sigma=5\%$. Pentagon markers indicate the iteration at which the stopping criterion~\eqref{eq:RC} is met.}                  \label{fig:deblur_rre}
\end{figure}   

\subsection{Tomographic Reconstruction Problems}
The following three experiments investigate tomographic reconstruction problems of increasing complexity. The second experiment considers an X-ray CT reconstruction problem, a standard test
  case in which the forward operator inherits the structure of the Radon transform. We then consider seismic travel-time tomography, characterized by a large sparse forward model arising from ray-based measurements. The final experiment addresses dynamic CT, where a sequence of time-dependent images must be reconstructed simultaneously using a space-time regularization model. These examples illustrate the effectiveness of the proposed approach for both static and dynamic tomographic inverse problems.
\subsubsection{Experiment 2: X-ray CT}\label{sec:exp_CT}
\begin{figure}[t]
\centering
\begin{subfigure}[t]{0.45\textwidth}
\centering
\includegraphics[width=\textwidth]{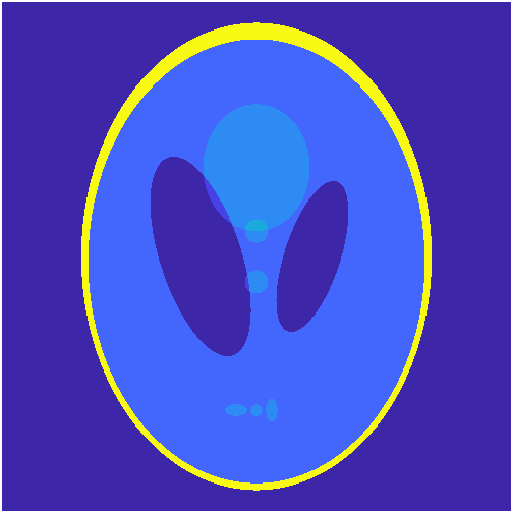}
\caption{}
\label{fig:ct_true}
\end{subfigure}
\hfill
\begin{subfigure}[t]{0.50\textwidth}
\centering
\includegraphics[width=\textwidth]{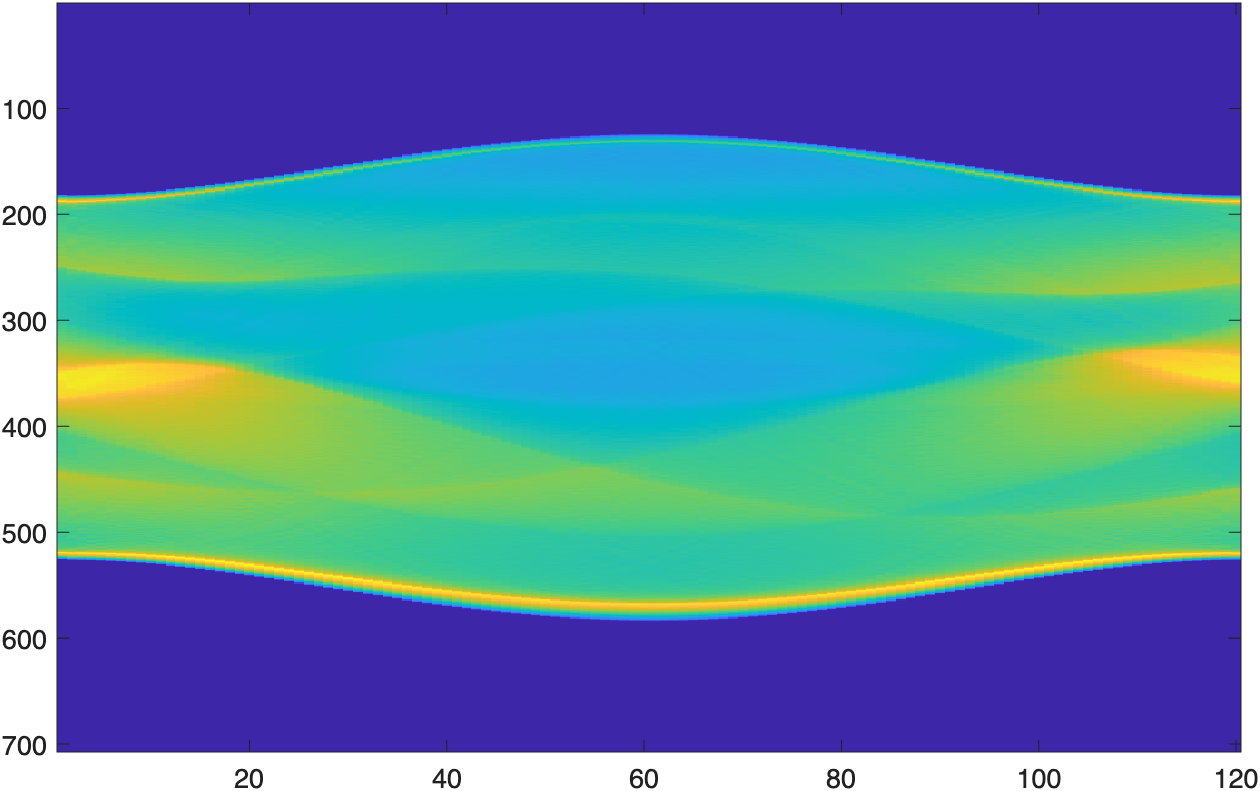}
\caption{}
\label{fig:ct_sinogram}
\end{subfigure}
\caption{Experiment 2 (Computerized tomography).
(a)~ True Shepp-Logan phantom.
(b)~Measured projections; rows correspond to detector positions,
columns to the $120$ projection angles
$\theta \in [0^\circ, 179^\circ]$.}
\label{fig:ct_setup}
\end{figure}

\begin{figure}[ht]
\centering
\begin{subfigure}[t]{0.48\textwidth}
\centering
\includegraphics[width=\textwidth]{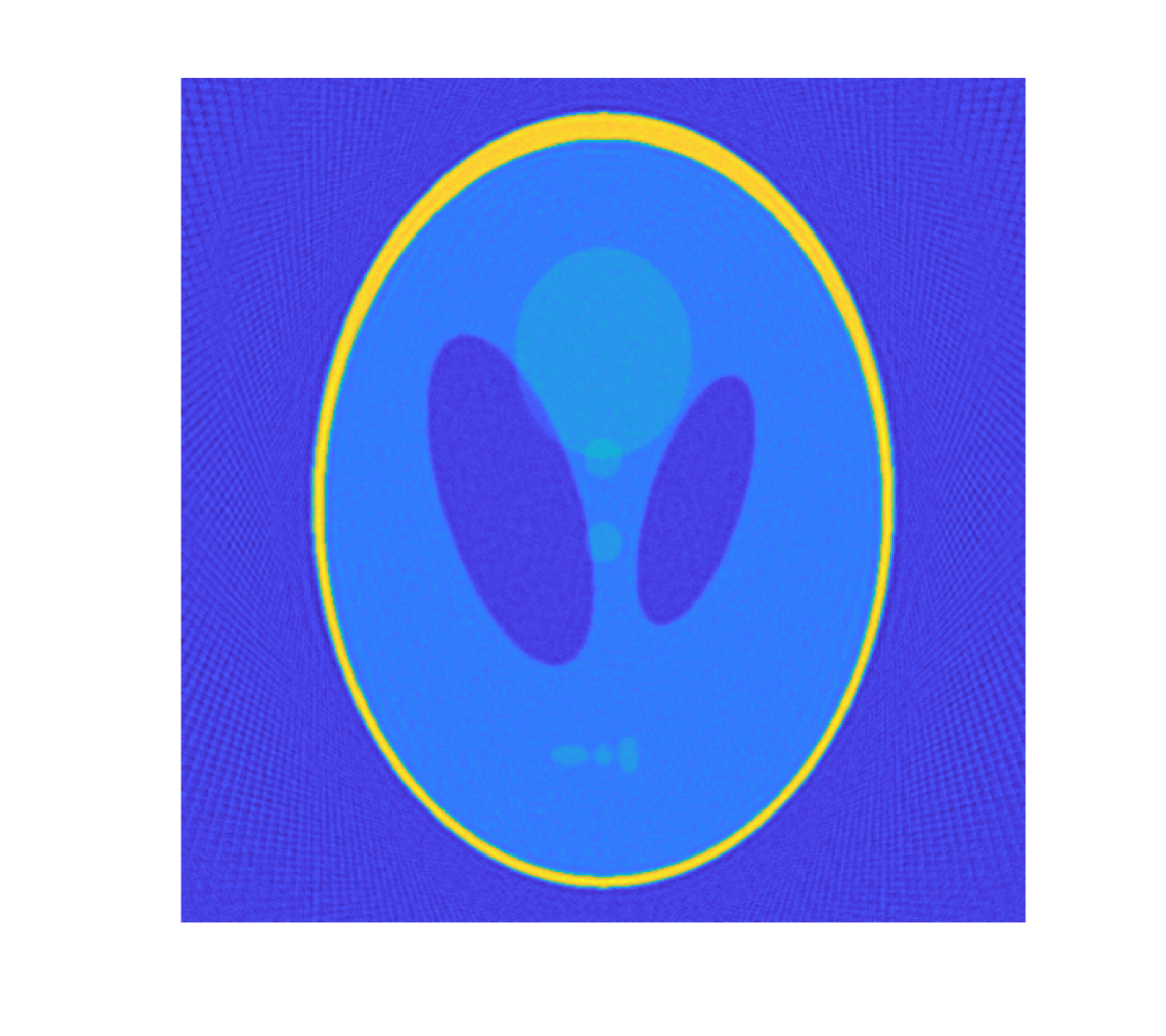}
\caption{}
\label{fig:ct_rre_time}
\end{subfigure}
\hfill
\begin{subfigure}[t]{0.48\textwidth}
\centering
\includegraphics[width=\textwidth]{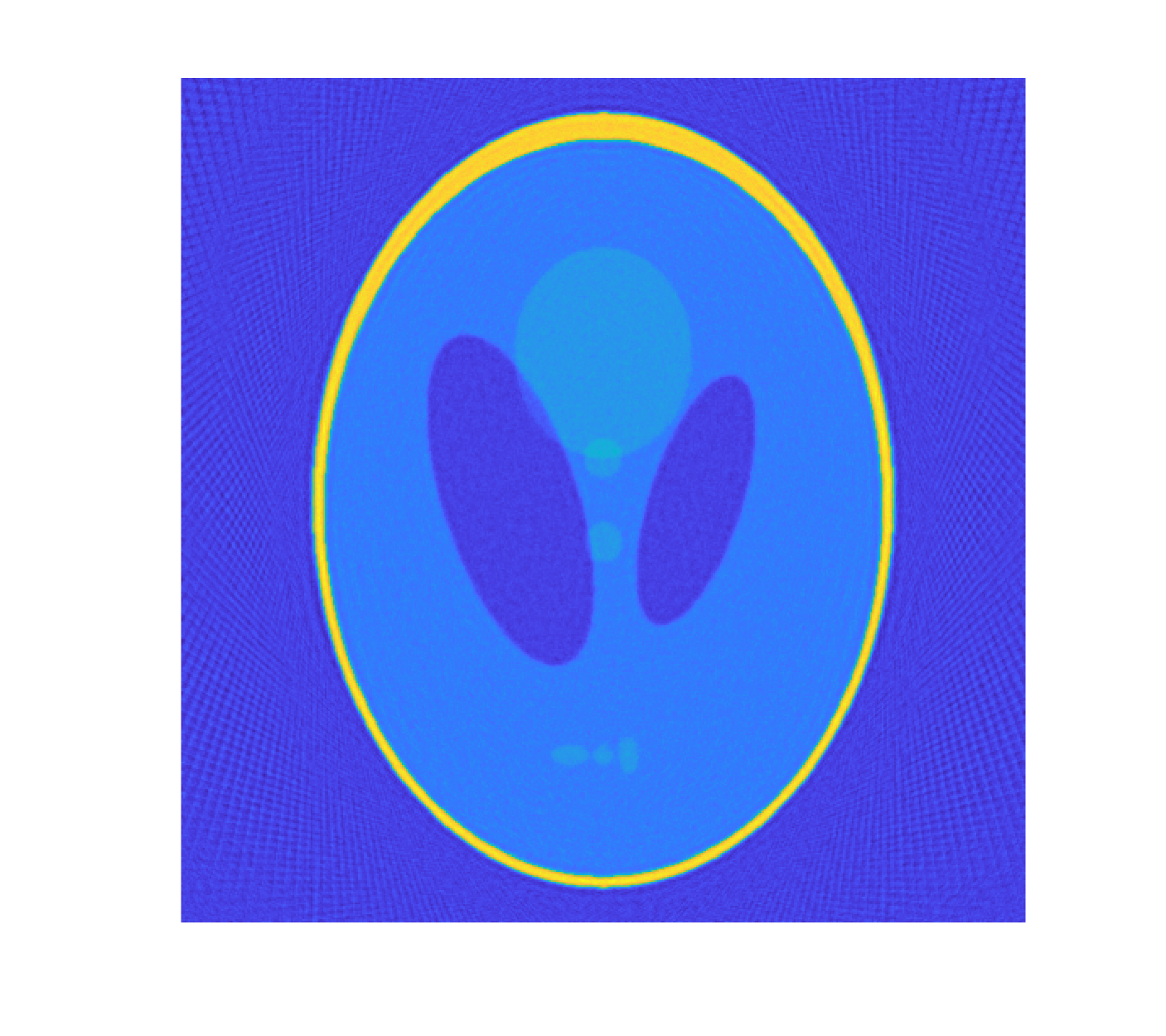}
\caption{}
\label{fig:ct_timing}
\end{subfigure}
\caption{Experiment 2 (Computerized Tomography with $n_x = n_y = 500$). (a) GKS reconstruction, (b) sGKS reconstruction.}
\label{fig:reconstructions}
\end{figure}

\begin{figure}[ht!]
\centering
\includegraphics[width=\textwidth]{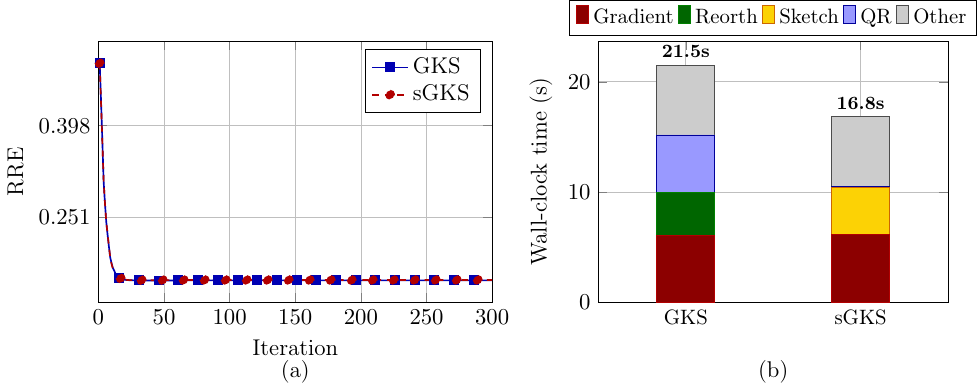}
\caption{Experiment 2 (Computerized tomography with $n_x = n_y = 500$ and optimal $\lambda$).      (a) RRE versus iteration for GKS and sGKS.       (b) total wall-clock time.}     \label{fig:CT_conv}\end{figure}
Having established the accuracy of sGKS across a range of sketch dimensions in the deblurring experiment, we now turn to a larger
problem where the computational savings become more pronounced.
We consider a two-dimensional parallel-beam X-ray CT reconstruction in which the goal is to recover an $n_x \times n_y$ cross-sectional image
from a set of line-integral measurements collected at multiple
projection angles. 
The forward operator $\bA \in \R^{m \times n}$ with $n = n_x n_y$
models parallel-beam projections at $120$ uniformly spaced angles in
$[0^\circ, 179^\circ]$, generated by \texttt{PRtomo} from
IRtools~\cite{GazzolaNagyPer19}.
We set $n_x = n_y = 500$, yielding $n = n_x n_y = 250{,}000$
unknowns and $m = 60{,}000$ measurements; the system is thus
underdetermined ($m \ll n$).
White Gaussian noise is added at relative level $\sigma = 1\%$.
The true image (Shepp-Logan phantom) is displayed in Figure~\ref{fig:ct_setup}(a) and the corresponding noisy sinogram (the image of measured projections), with rows indexed by detector position and columns by projection angle is shown in Figure~\ref{fig:ct_setup}(b). 
In this experiment we use the regularization operator~$\bL$ in~\eqref{eq: reg_matrix}
with $p = 499{,}000$ rows. 

We compare GKS and sGKS, both run for $300$ iterations with fixed regularization parameter.
The latter is hand-tuned to
minimize the RRE; both methods use the same value, $\lambda^* = 35.9$. We use ~$s\equiv s_A=s_L = 606$.
Figure~\ref{fig:reconstructions} displays the reconstructions produced by
both methods at the optimal regularization parameter.
The images are visually indistinguishable, and the final RRE values  are nearly identical ($0.1831$ for GKS and $0.1834$ for sGKS), confirming that the computational savings do not come at the expense of reconstruction quality.
Figure~\ref{fig:CT_conv}(a) plots the RRE as a function of the
iteration count for both methods.
The convergence trajectories are nearly identical: both GKS and sGKS
converge monotonically to a final RRE of approximately $0.183$.
This confirms that the sketching approximation does not
degrade the quality of the projected regularization problem, even on
a system with $n = 250{,}000$ unknowns and a regularization operator
with nearly half a million rows. 
Figure~\ref{fig:CT_conv}(b) reports the total wall-clock time
and a stacked-bar decomposition of the per-component costs.
For GKS, the two dominant operations are the incremental QR
factorization update of
$\bL \bV_k \in \R^{p \times d_k}$ (with $p = 499{,}000$) and the
double modified Gram-Schmidt reorthogonalization of each new basis
vector against all previous vectors.
Together, these account for more than $40\%$ of the total GKS runtime.

In sGKS, the QR update operates on the sketched matrix $\bS_L \bL \bV_k \in 
\R^{s \times d_k}$ rather than the full $p \times d_k$ matrix, reducing its cost by a factor proportional to $p / s$, and the reorthogonalization is      
  eliminated entirely. In its place, the SRTT sketch application adds a modest 
  overhead that is more than offset by these savings. As mentioned at the end of
   Section~\ref{The sketching operator}, this overhead is partly due to our
  naive implementation, which causes the application of $\bS$ to an
  $n$-dimensional vector to cost $\mathcal{O}(n\log n)$ flops; more
  sophisticated implementations would reduce this to $\mathcal{O}(n\log s)$
  flops~\cite{WOOLFEetal2008}, inducing a further speedup.

\subsubsection{Experiment 3: Seismic Travel-Time Tomography}\label{sec:seismic}

The third experiment considers seismic travel-time tomography, in which
the goal is to reconstruct the subsurface slowness field from measured
arrival times of seismic waves traveling between sources and receivers
along the boundary of the domain.
Under the ray approximation, each measurement is a line integral of the
slowness field along a ray path, producing a sparse forward operator
$\bA \in \R^{m \times n}$ whose rows correspond to source--receiver
pairs.
This problem differs from the previous experiments in two respects:
the operator is sparse and unstructured matrix-vector products
require explicit sparse arithmetic rather than fast transforms and
the system is overdetermined ($m \gg n$).

More specifically, we consider the recovery of a two-dimensional tectonic velocity model
on an $n_x \times n_y$ grid with $n_x = n_y = 256$, yielding
$n = n_x n_y = 65{,}536$ unknowns.
The forward operator
$\bA \in \R^{m \times n}$ is a ray-based travel-time matrix generated
by \texttt{PRseismic} from IRtools~\cite{GazzolaNagyPer19} with
$2n_x$ receivers and $3n_x$ sources uniformly distributed along the
boundary, giving $m = 393{,}216$ measurements ($m \approx 6n$).
White Gaussian noise is added at relative level $\sigma = 1\%$.
In this experiment we use the regularization operator~$\bL$ with
$p = 130{,}560$ rows.
The true velocity model and the noisy travel-time data are displayed
in Figure~\ref{fig:seis_setup}.

In addition to comparing GKS and sGKS, this experiment includes
two methods from the hybrid Krylov framework: LSQR-opt, the hybrid
LSQR method based on the Golub--Kahan
bidiagonalization~\cite{ChungGazzola25,GazzolaNagyPer19}, and
rLSQR-opt, its randomized counterpart based on the randomized
Golub--Kahan bidiagonalization with SRTT
sketching~\cite{ChungGazzola25}.
As in the deblurring experiment, we vary the sketch dimension~$s$
for sGKS across multiple values.
More precisely, we compare the following methods:
\begin{enumerate}
  \item [$\diamond$] \textbf{GKS}: the standard method with double modified
        Gram--Schmidt reorthogonalization and incremental QR updates;
  \item [$\diamond$] \textbf{LSQR-opt}: hybrid LSQR based on the Golub--Kahan
        bidiagonalization with fixed regularization parameter
        (\texttt{IRhybrid\_lsqr} from
        IRtools~\cite{GazzolaNagyPer19});
  \item [$\diamond$] \textbf{rLSQR-opt}: its randomized variant based on the
        randomized Golub--Kahan bidiagonalization with SRTT sketching
        and fixed regularization parameter~\cite{ChungGazzola25};
  \item [$\diamond$] \textbf{sGKS}: at sketch dimensions
        $s = 300$, $1{,}000$, $5{,}000$, and $10{,}000$.
\end{enumerate}
All methods are run for $300$ iterations.
GKS and sGKS use the same hand-tuned regularization parameter
$\lambda^* \approx 37.3$.
For LSQR-opt, the regularization parameter is selected from $15$
logarithmically spaced candidates $\lambda \in [10^{-1}, 10^3]$;
for rLSQR-opt, a joint search over $\lambda$ and sketch dimension
$s \in \{500, 1000, 2000, 5000\}$ is performed.
In both cases the value minimizing the RRE is retained.

Figure~\ref{fig:seis_recon} displays the reconstructions for all five
representative methods at their respective optimal parameters.
GKS and sGKS with $s = 10{,}000$ produce visually indistinguishable
results with well-resolved velocity boundaries.
sGKS with $s = 300$ shows mild smoothing due to the limited sketch
dimension.
LSQR-opt and rLSQR-opt recover the overall structure but with
slightly less contrast in the fine-scale features.

Figure~\ref{fig:seis_conv} reports the RRE as a function of iteration
count for all methods.
GKS achieves the lowest final RRE ($\approx 0.097$), and sGKS with
the largest sketch dimensions ($s = 5{,}000$ and $s = 10{,}000$)
closely tracks the GKS trajectory, reaching final RREs of $0.1029$
and $0.1031$, respectively.
As the sketch dimension decreases, the sGKS curves gradually deviate
from GKS: $s = 1{,}000$ reaches $0.111$ and
$s = 300$ saturates at $0.150$.
This is consistent with the deblurring experiment and confirms that
the accuracy and efficiency tradeoff controlled by~$s$ is
problem-independent.
LSQR-opt converges to a final RRE of $0.117$, and rLSQR-opt to
$0.117$, both with their respective hand-tuned regularization
parameters.
Neither hybrid method matches the GKS/sGKS accuracy on this problem,
which reflects the different projected regularization structure:
GKS solves a general-form Tikhonov problem with the explicit
regularization operator~$\bL$, whereas the LSQR-based methods
regularize in the Krylov coordinate system.

Figure~\ref{fig:seis_timing} presents a stacked-bar decomposition of
the wall-clock time for all methods.
For GKS and sGKS, the gradient computation dominates, as the sparse
matrix-vector products with
$\bA \in \R^{393{,}216 \times 65{,}536}$ are inherently expensive.
The QR and reorthogonalization costs in GKS are a smaller fraction
of the total than in the CT experiment because the regularization
operator $\bL$ has fewer rows ($p = 130{,}560$ vs.\ $499{,}000$),
and consequently the sGKS variants achieve only a modest speedup
(total times $\approx 27$\,s vs.\ $30$\,s for GKS).
LSQR-opt is the most expensive method at $60$\,s, reflecting the
cost of full reorthogonalization in the Golub--Kahan process together
with the projected Tikhonov solve at each iteration.
rLSQR-opt, which replaces the reorthogonalization with sketched
inner products, reduces the running time to $32$\,s.

On this problem, the QR and reorthogonalization costs that sGKS eliminates
account for only ${\sim}15\%$ of the GKS total, with the gradient computation
and matrix--vector products constituting the remaining ${\sim}85\%$, so sGKS
provides a modest speedup over GKS (${\sim}10\%$). Nevertheless, the
comparison
with LSQR-opt ($60$\,s) and rLSQR-opt ($32$\,s) highlights that the
GKS/sGKS framework is substantially more efficient on this problem while also
delivering superior reconstruction quality. A more pronounced speedup is
demonstrated in Experiment~4 (Section~\ref{sec:dynamic_CT}), where the large
space--time regularization operator makes the QR and reorthogonalization costs
dominant, and sGKS achieves a $3\times$ speedup over standard GKS.

\begin{figure}[t]
\centering
\includegraphics[width=0.85\textwidth]{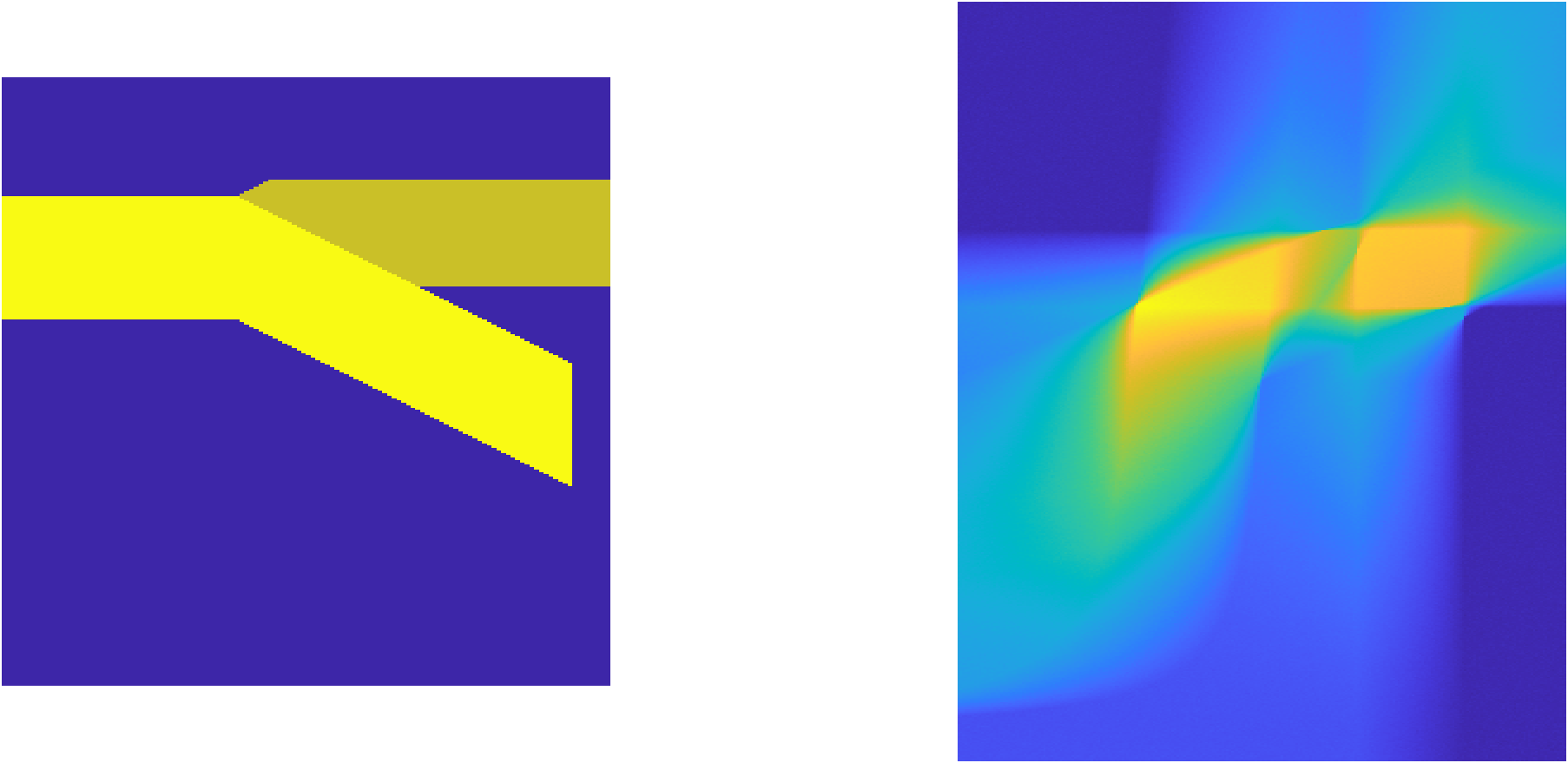}
\caption{Experiment 3 (Seismic tomography).
(a)~Tectonic velocity model ($n_x = n_y = 256$).
(b)~Noisy travel-time measurements ($\sigma = 1\%$).}
\label{fig:seis_setup}
\end{figure}

\begin{figure}[t]
\centering
\begin{subfigure}[t]{0.19\textwidth}
\centering
\includegraphics[width=\textwidth]{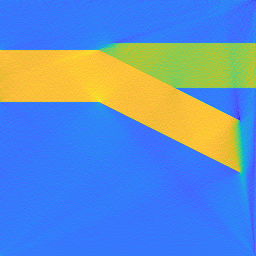}
\caption{LSQR-opt}
\end{subfigure}
\hfill
\begin{subfigure}[t]{0.19\textwidth}
\centering
\includegraphics[width=\textwidth]{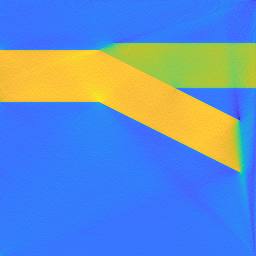}
\caption{rLSQR-opt}
\end{subfigure}
\hfill
\begin{subfigure}[t]{0.19\textwidth}
\centering
\includegraphics[width=\textwidth]{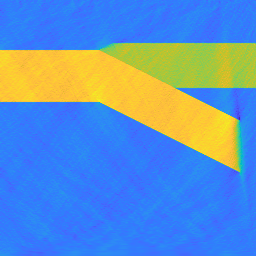}
\caption{GKS}
\end{subfigure}
\hfill
\begin{subfigure}[t]{0.19\textwidth}
\centering
\includegraphics[width=\textwidth]{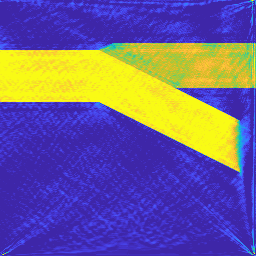}
\caption{sGKS $s{=}300$}
\end{subfigure}
\hfill
\begin{subfigure}[t]{0.19\textwidth}
\centering
\includegraphics[width=\textwidth]{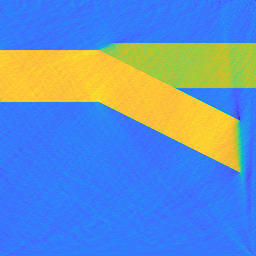}
\caption{sGKS $s{=}10000$}
\end{subfigure}
\caption{Experiment 3 (Seismic tomography,
$n_x = n_y = 256$): reconstructions at optimal regularization
parameters provided by the different methods we tested.}
\label{fig:seis_recon}
\end{figure}

\begin{figure}[t]
  \centering
\includegraphics[width=\textwidth]{  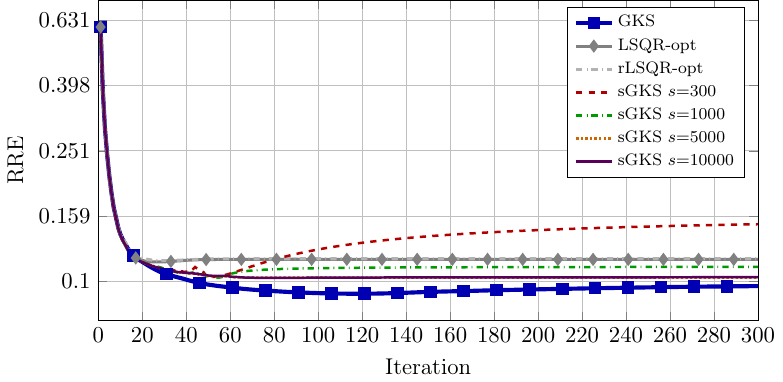}
  \caption{Experiment 3 (Seismic tomography):
           RRE versus iteration at the optimal~$\lambda^*$ for GKS,
           LSQR-opt, rLSQR-opt, and sGKS at four sketch dimensions.}
\label{fig:seis_conv}
\end{figure}

\begin{figure}[t]
  \centering
\includegraphics[width=\textwidth]{   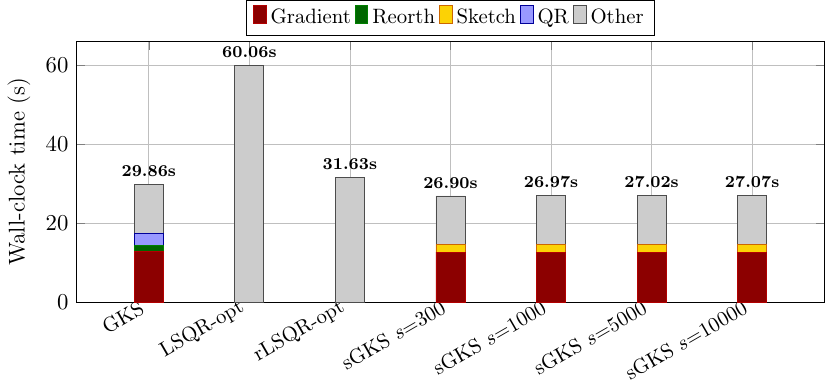}
  \caption{Experiment 3 (Seismic tomography):
           stacked-bar wall-time breakdown for GKS, sGKS at
           each sketch dimension, LSQR-opt, and rLSQR-opt. For GKS and sGKS, the gradient cost dominates both methods so that the QR and
           reorthogonalization savings provided by sGKS are modest on this problem.}
  \label{fig:seis_timing}
\end{figure}

\subsection{Experiment 4: Dynamic CT}\label{sec:dynamic_CT}
The final experiment demonstrates the applicability of sGKS to a
dynamic inverse problem, where the unknown is a sequence of images
evolving over time and the regularization operator encodes both spatial
and temporal structure.
Unlike the previous experiments, which reconstruct a single static
image, this test requires the simultaneous recovery of a time-varying
object from temporally fragmented measurements.
The difficulty is twofold: (i)~each individual time frame is observed
from only a narrow range of projection angles, making every
single-frame subproblem severely ill-posed, and (ii)~the full
space--time system is very large, so the per-iteration cost of standard
GKS becomes a significant practical concern.
We consider a dynamic parallel-beam CT problem in which the unknown is
a time-varying phantom consisting of four piecewise-constant geometric
objects (two L-shaped and two T-shaped blocks) with distinct
intensities.
Each object translates rigidly in a different direction at one of two
speeds ($1$ or $2$ pixels per frame).
The phantom is defined on an $n_x \times n_y = 256 \times 256$ spatial
grid and evolves over $n_t = 10$ time frames.
The test phantom is inspired by the image sequences
in~\cite{okunola2025efficient}.
The forward operator at each time frame~$t$ is a parallel-beam Radon
transform $\bA_t \in \R^{m_t \times n_s}$ ($n_s = n_x^2$) generated
by \texttt{PRtomo} from IRtools~\cite{GazzolaNagyPer19}.
At each time step~$t$, a distinct set of $30$ projection angles is
used, arranged in disjoint angular sectors: the angular range
$[0^\circ, 180^\circ)$ is partitioned into $n_t$ equal sectors, and
frame~$t$ uses only the angles in its assigned sector.
Frame~$t$ thus observes only an $18^\circ$ range of projections.
The full $180^\circ$ angular coverage is achieved only when all
$n_t$ frames are considered jointly, which is the key motivation for
the space--time formulation: temporal coupling allows each frame to
borrow angular information from neighboring frames.

The block-diagonal forward operator
$\bA = \operatorname{blkdiag}(\bA_1, \ldots, \bA_{n_t})
\in \R^{m \times n}$
has $n = n_x n_y n_t = 655{,}360$ unknowns and
$m = 153{,}600$ measurements.
The system is heavily underdetermined ($m \ll n$).
White Gaussian noise is added at relative level $\sigma = 0.1\%$.
The true frames and corresponding sinograms at selected time steps are
displayed in Figure~\ref{fig:dyntomo_true_sino}.

The full problem can be formulated as
\begin{equation}\label{eq:blockF}
\underbrace{
\begin{bmatrix} \bA_{1} \\
 & \ddots \\
 & & \bA_{n_t}
\end{bmatrix}}_{\bA}
\underbrace{\begin{bmatrix}
    \bx_1 \\ \vdots \\ \bx_{n_t}
\end{bmatrix}}_{\bx}
+
\underbrace{\begin{bmatrix}
    \be_1 \\ \vdots \\ \be_{n_t}
\end{bmatrix}}_{\be}
=
\underbrace{\begin{bmatrix}
    \bb_1 \\ \vdots \\ \bb_{n_t}
\end{bmatrix}}_{\bb}.
\end{equation}
Regularization leads to the minimization problem
\begin{equation}\label{eq:spacetime_tik}
    \min_{\bx \in \R^{n}} \|\bA\bx - \bb\|_2^2 + \lambda\|\Psi\bx\|_2^2,
\quad \text{with} \quad \Psi = \begin{bmatrix}
  \bI_{n_t}\otimes\bI_{n_y}\otimes\bL_{x} \\
  \bI_{n_t}\otimes \bL_{y}\otimes\bI_{n_x} \\
  \bL_{t}\otimes\bI_{n_y}\otimes\bI_{n_x}
\end{bmatrix},
\end{equation}
where $\bL_d$ denotes the discrete first-derivative operator in
direction~$d$, for $d = x$ (vertical), $d = y$ (horizontal), and
$d = t$ (time).
The first two blocks enforce spatial smoothness within each frame;
the third enforces temporal consistency across frames.
Alternative definitions of~$\Psi$ for dynamic inverse problems can be
found in~\cite{pasha2021efficient} and, within the Bayesian framework,
in~\cite{lan2025spatiotemporal}.

For this experiment we also test a version of sGKS equipped with iterative refinement (IR); see the discussion at the end of section~\ref{sec:analysis}. We follow the paradigm named FOSSILS in~\cite{Iterative_refinement_sketching}. In particular, IR is implemented by applying~\cite[Algorithm 1]{Iterative_refinement_sketching} with two different settings: (i) $q=0$ means that line 6 in~\cite[Algorithm 1]{Iterative_refinement_sketching} is not performed; (ii) $q=1$ means that a single iteration of~\cite[Algorithm 2]{Iterative_refinement_sketching} is performed in line 6 in~\cite[Algorithm 1]{Iterative_refinement_sketching}. 

We thus consider the following methods for comparison:
\begin{enumerate}
  \item [$\diamond$] \textbf{OBO}: one-by-one GKS reconstruction of each frame
        independently using only the spatial regularization
        operator~$\bL$ in~\eqref{eq: reg_matrix} (no temporal coupling);
  \item [$\diamond$] \textbf{GKS}: standard GKS with the full space-time
        operator~$\Psi$;
  \item [$\diamond$] \textbf{sGKS (no IR)}: the sketched variant without
        iterative refinement;
  \item [$\diamond$] \textbf{sGKS (IR, $q=0$)}: sGKS with IR as in \cite[Algorithm 1]{Iterative_refinement_sketching} avoiding line 6 in the latter scheme;
  \item [$\diamond$] \textbf{sGKS (IR, $q=1$)}: sGKS with IR as in \cite[Algorithm 1]{Iterative_refinement_sketching} where line 6 consists of a single iteration of  \cite[Algorithm 2]{Iterative_refinement_sketching}.
\end{enumerate}
All space-time methods run for $400$ iterations and we fix $s= 806$.
The regularization parameter for each method is hand-tuned to minimize
the RRE: $\lambda^*_{\mathrm{GKS}} = 1.90$,
$\lambda^*_{\text{sGKS}} = 6.81$,
$\lambda^*_{\text{sGKS-IR}} = 1.90$ (both $q = 0$ and $q = 1$),
and $\lambda^*_{\mathrm{OBO}} = 146.8$.
The OBO baseline requires a regularization parameter roughly two
orders of magnitude larger because, without temporal coupling, each
frame must rely entirely on spatial regularization to compensate for
the severe angular undersampling.
Both IR variants recover the same optimal~$\lambda^*$ as GKS,
confirming that even a single refinement step corrects the sketching
distortion in the projected solve.

Figure~\ref{fig:dyntomo_recon} displays selected reconstructed frames
(frames~$1$, $3$, $7$, and~$10$) for each method.
The space-time methods produce visually sharper reconstructions with
well-defined object boundaries.
The OBO reconstructions show visible artifacts of limited-angle tomography because the angular gaps in each
frame's measurement set leave certain spatial frequencies unobserved. GKS and all the sGKS variants produce very similar results, for all frames.

Figure~\ref{fig:dyntomo_conv} plots the RRE as a function of
iteration count.
The most striking feature is the clear separation between the
space-time methods and the OBO baseline: all four space--time solvers
converge to an RRE of approximately $0.41$, while OBO saturates
at $0.56$.
Among the space-time methods, GKS and both IR variants
follow essentially the same trajectory, confirming that iterative
refinement fully corrects the sketching error at every step.
sGKS without IR converges faster but stagnates to a slightly higher RRE ($0.43$), which
is expected since it operates at a different optimal~$\lambda^*$ and
the sketching distortion accumulates without correction.

Figure~\ref{fig:dyntomo_timing} provides the total wall-clock times
and a stacked-bar decomposition of the per-component costs. GKS requires $99$\,s, dominated by the QR factorization update on   the large space-time regularization operator~$\Psi$ and the double  modified Gram--Schmidt reorthogonalization.                       sGKS (no IR) is the fastest at $33$\,s, achieving a $3\times$ speedup by replacing both expensive operations with their sketched  counterparts. The IR variants are slower ($325$\,s and $473$\,s for $q = 0$ and $q = 1$ respectively), as each refinement step requires forming and solving the full projected normal equations. The OBO baseline takes $100$\,s (the sum over $10$ independent per-frame solves), comparable to the space-time methods but at significantly worse reconstruction quality.  

We now turn our attention to the impact that IR can have on the properties of the basis computed by sGKS; cf. the discussion in section~\ref{sec:analysis}.
Figure~\ref{fig:dyntomo_sv_crossgram}(a) compares the singular values
$\sigma_j(\bW_k)$ of the sGKS basis after $400$ iterations for the
three sGKS variants.
Without IR, the final basis $\bW_k$ computed by sGKS is numerically low-rank. It is interesting to notice that the rank of $\bW_k$ is approximately 150, that pretty much matches the iteration index in which the RRE provided by sGKS without IR starts stagnating (cf. Figure~\ref{fig:dyntomo_conv}). Unveiling a possible connection between these two elements would be interesting but also beyond the scope of this paper.  

With IR ($q = 0$ or $q = 1$), the singular values of $\bW_k$ remain close to
one throughout, confirming that the refinement step restores the
orthogonality of the basis to a certain extent.
This also explains why the IR variants match the GKS convergence
trajectory and select the same optimal~$\lambda^*$: the solution of the sketched projected problem
gets very close to the solution of the unsketched one after refining.

Finally, to assess how well the sGKS basis captures the
subspace generated by standard GKS, we compare the two sets of basis vectors.
Let $\bV_{\mathrm{GKS}}\in\mathbb{R}^{n\times d}$ denote the
orthonormal Krylov basis produced by GKS after $d=400$ iterations, and
let $\bW_{\mathrm{sGKS}}\in\mathbb{R}^{n\times d}$ denote the
(non-orthonormal) basis produced by sGKS with sketch dimension
$s=10{,}000$ and no IR.
We compute the economy QR factorization
$\bQ\bR = \mathrm{qr}(\bW_{\mathrm{sGKS}})$, obtaining an
orthonormal factor $\bQ$, and form the cross-Gram matrix
$|\bQ^\top\bV_{\mathrm{GKS}}|$ whose components are the absolute value of the cosine of the angle $\theta_{i,j}$ between the $i$th column of $\bQ$ and the $j$th column of $\bV_{\mathrm{GKS}}$. These values are reported in Figure~\ref{fig:dyntomo_sv_crossgram} (b) which shows a strong diagonal band, confirming that the sGKS basis vectors are well-aligned with the corresponding GKS vectors for the first ${\sim}250$ iterations. The near-identity structure in the upper-left block indicates that the the subspace generated by sGKS preserves the leading subspace directions of $\text{range}(\bV_{\mathrm{GKS}})$ almost exactly, regardless the lack of any orthogonalization. Beyond index ${\sim}250$, the diagonal broadens and
off-diagonal entries become visible, indicating a gradual loss of
alignment for later basis vectors. However, at this point, the RRE provided by both GKS and sGKS had stopped decreasing in a significant matter; cf. Figure~\ref{fig:dyntomo_conv}.

\begin{figure}[t!]
\centering
\begin{tabular}{cccc}
$t =1$ & \qquad \qquad \qquad $t =3$ & \qquad \qquad \qquad $t =7$ & \qquad \qquad \qquad $t =10$ \\
\end{tabular}
\begin{tabular}{cccc}
\includegraphics[height = 0.2\textwidth, width = .2\textwidth]{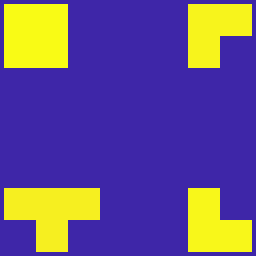} &
\includegraphics[height = 0.2\textwidth, width = .2\textwidth]{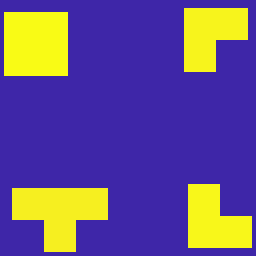} &
\includegraphics[height = 0.2\textwidth, width = .2\textwidth]{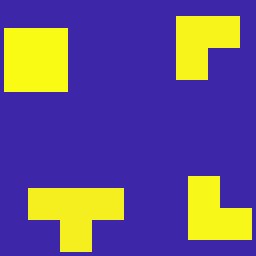} &
\includegraphics[height = 0.2\textwidth, width = .2\textwidth]{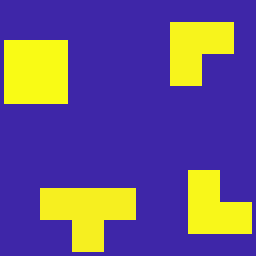}
\end{tabular}
\begin{tabular}{cccc}
\includegraphics[height = 0.05\textwidth, width = .2\textwidth]{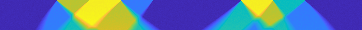} &
\includegraphics[height = 0.05\textwidth, width = .2\textwidth]{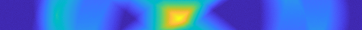} &
\includegraphics[height = 0.05\textwidth, width = .2\textwidth]{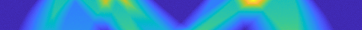} &
\includegraphics[height = 0.05\textwidth, width = .2\textwidth]{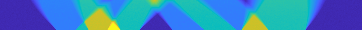}
\end{tabular}
\caption{Experiment 4 (Dynamic CT). First row: true images
($n_x = n_y = 256$) at time steps $t=1, 3, 7, 10$. Second row:
sinograms at the same time steps.}
\label{fig:dyntomo_true_sino}
\end{figure}

\begin{figure}[ht!]
\centering
\begin{tabular}{cccc}
$t =1$ & \qquad \qquad \qquad $t =3$ & \qquad \qquad \qquad $t =7$ & \qquad \qquad \qquad $t =10$ \\
\end{tabular}
\begin{tabular}{cccc}
\includegraphics[height = 0.2\textwidth, width = .2\textwidth]{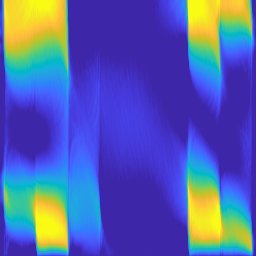} &
\includegraphics[height = 0.2\textwidth, width = .2\textwidth]{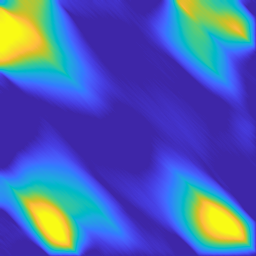} &
\includegraphics[height = 0.2\textwidth, width = .2\textwidth]{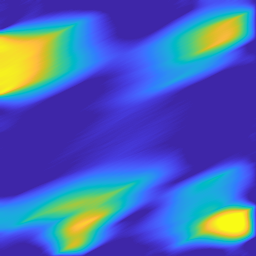} &
\includegraphics[height = 0.2\textwidth, width = .2\textwidth]{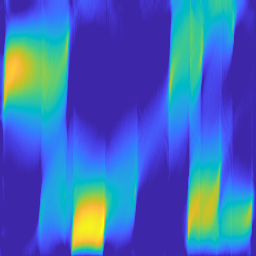}
\end{tabular}
\begin{tabular}{cccc}
\includegraphics[height = 0.2\textwidth, width = .2\textwidth]{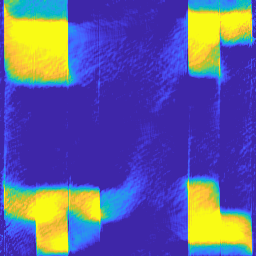} &
\includegraphics[height = 0.2\textwidth, width = .2\textwidth]{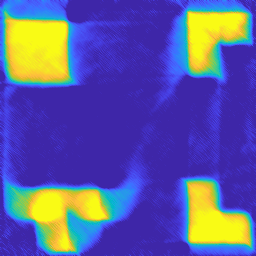} &
\includegraphics[height = 0.2\textwidth, width = .2\textwidth]{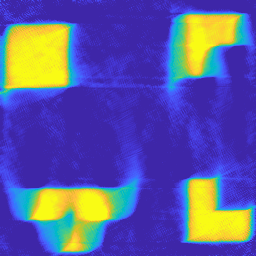} &
\includegraphics[height = 0.2\textwidth, width = .2\textwidth]{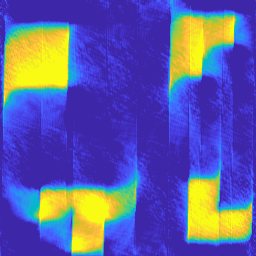}
\end{tabular}
\begin{tabular}{cccc}
\includegraphics[height = 0.2\textwidth, width = .2\textwidth]{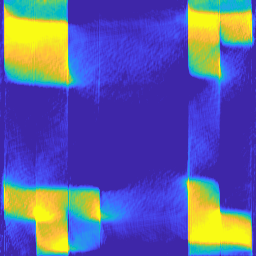} &
\includegraphics[height = 0.2\textwidth, width = .2\textwidth]{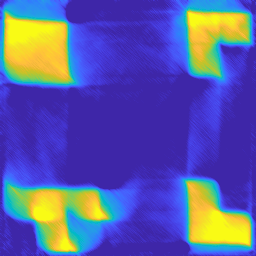} &
\includegraphics[height = 0.2\textwidth, width = .2\textwidth]{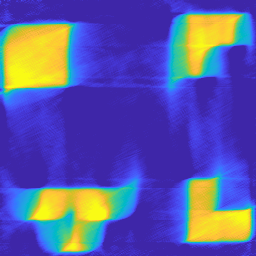} &
\includegraphics[height = 0.2\textwidth, width = .2\textwidth]{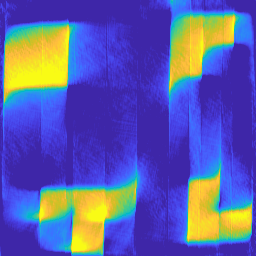}
\end{tabular}
\begin{tabular}{cccc}
\includegraphics[height = 0.2\textwidth, width = .2\textwidth]{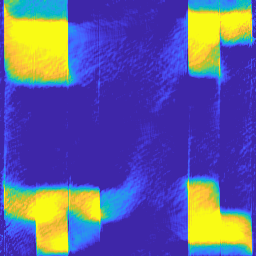} &
\includegraphics[height = 0.2\textwidth, width = .2\textwidth]{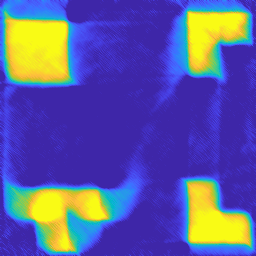} &
\includegraphics[height = 0.2\textwidth, width = .2\textwidth]{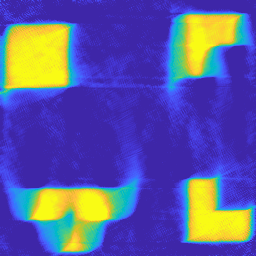} &
\includegraphics[height = 0.2\textwidth, width = .2\textwidth]{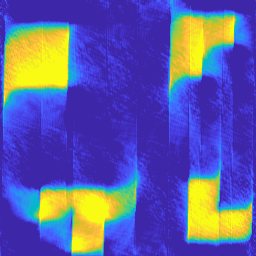}
\end{tabular}
\caption{Experiment 4 (Dynamic CT): reconstructed frames at time steps
$1, 3, 7, 10$. From top to bottom: OBO, GKS, sGKS (no IR),
sGKS (IR, $q=1$).}
\label{fig:dyntomo_recon}
\end{figure}

\begin{figure}[t]
  \centering
\includegraphics[width = \textwidth]{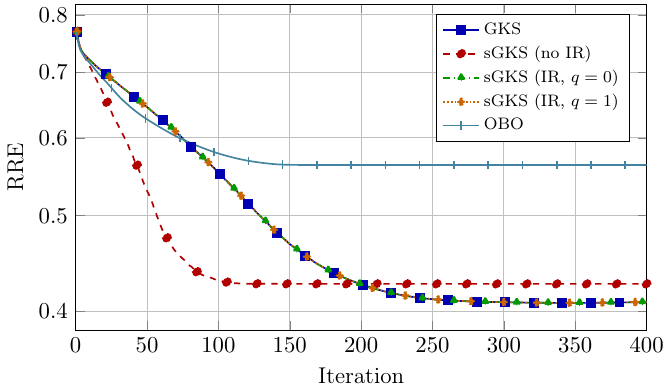}
  \caption{Experiment 4 (Dynamic CT,
           $256 \times 256 \times 10$ frames,
           $\sigma = 0.1\%$).
           RRE versus iteration at the optimal~$\lambda^*$
           for each method.
           }
  \label{fig:dyntomo_conv}
\end{figure}

\begin{figure}[t]
\centering
\includegraphics[width = \textwidth]{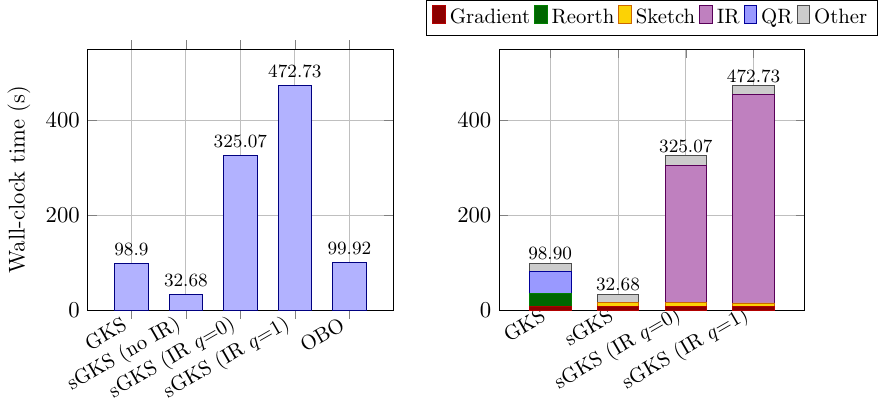}
\caption{Experiment 4 (Dynamic CT).
(a)~Total wall-clock time for $400$ iterations (all five methods including OBO).
(b)~Stacked-bar breakdown of the per-component costs for the four space-time methods.}
\label{fig:dyntomo_timing}
\end{figure}

\begin{figure}[t]                     \centering
\includegraphics[width = \textwidth]{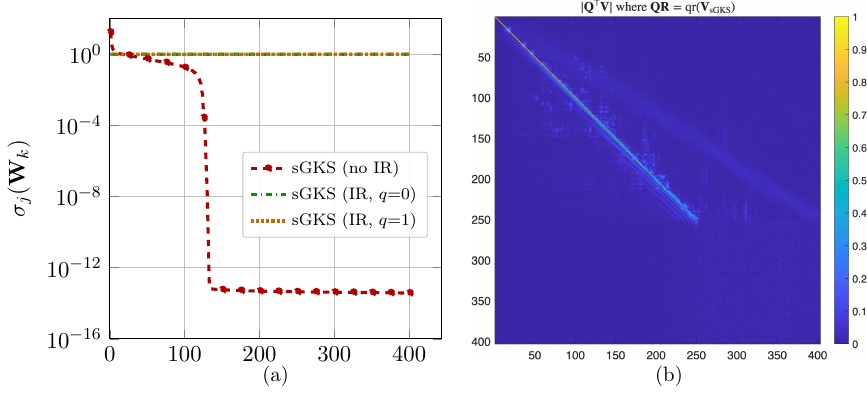}
\caption{Experiment 4 (Dynamic CT).           (a)~Singular values $\sigma_j(\mathbf{W}_k)$ of the sGKS basis
    after $k = 400$ iterations: without iterative refinement (IR) the           
    singular values decay to machine precision, while both IR variants           preserve near-unit values.                 (b)~Absolute cross-Gram matrix $|\mathbf{Q}^\top                           \mathbf{V}_{\mathrm{GKS}}|$,                  where $\mathbf{Q}\mathbf{R} = \mathrm{qr}(\mathbf{W}_{\mathrm{sGKS}})$   with $s = 10{,}000$; the dominant diagonal band confirms that the           
    leading sGKS basis vectors closely align with their GKS counterparts.} 
\label{fig:dyntomo_sv_crossgram}    \end{figure}
\section{Conclusions}\label{sec:conclusions}
We have presented sGKS, a sketched variant of the generalized Krylov subspace method for large-scale Tikhonov regularization with a general regularization matrix. The method completely skips any explicit orthogonalization of the basis and applies randomized dimensionality reduction to the thin QR factorizations of the projected matrices.

The sketched QR factorizations operate on compressed matrices whose row dimension equals the sketch size rather than the ambient dimensions of the forward operator and the regularization matrix, and are maintained incrementally via rank-one column appends, yielding substantial per-iteration savings. 

A central observation underlying our approach is that explicit reorthogonalization of the basis is not essential for the construction of effective approximation subspaces. Indeed, no step of the GKS iteration relies intrinsically on the orthogonality of the basis, in contrast to classical Krylov methods where the Arnoldi or Lanczos relation ties the projected problem directly to the orthogonalization process. Dropping reorthogonalization eliminates what is often the dominant cost per iteration without compromising the quality of the search space, provided that the newly computed gradient directions continue to enrich the subspace. 

The resulting algorithm is independent of the choice of the sketching operator and preserves the approximation quality of the original GKS method, as guaranteed by the oblivious subspace embedding property.

We have shown that, in the absence of sketching in the projected solve, the proposed method produces iterates that are mathematically identical to those of standard GKS. When the projected least-squares problem is solved in sketched form, the method delivers quasi-optimal residual norms whose deviation from the unsketched solution is controlled by the embedding quality of the sketching operators. Furthermore, we have observed that on more challenging problems where the loss of numerical rank in the basis becomes significant, iterative refinement of the projected solve restores the spectral properties of the basis and recovers the full accuracy of the unsketched method. This benefit, however, comes at a much higher computational cost.

Numerical experiments on image deblurring, X-ray computerized tomography, seismic travel-time tomography, and dynamic computerized tomography confirm that sGKS matches the reconstruction accuracy of the standard GKS method while significantly reducing both per-iteration cost and overall wall-clock time. 

Several directions remain open for future work. Adaptive selection of the sketch dimension during the iteration, rather than fixing it a priori, could further reduce overhead in the early iterations when the subspace dimension is small. Extending the framework to nonlinear inverse problems and to regularizers beyond quadratic penalties, such as those arising in mixed-norm and total variation formulations, is another natural direction. Finally, a distributed-memory implementation exploiting the reduced communication footprint of the sketched operations could enable the solution of truly large-scale three-dimensional problems.

\section*{Acknowledgments}
MP acknowledges support from the National Science Foundation (NSF) under
Grant No.\ DMS-2410699. Any opinions, findings, conclusions, or
recommendations expressed in this material are those of the authors and do not
necessarily reflect the views of the National Science Foundation.
Both authors acknowledge support from the NSF under Grant No. DMS-1929284 while the authors were in residence at the Institute for Computational and Experimental Research in Mathematics in Providence, RI, during the \emph{Stochastic and Randomized Algorithms in Scientific Computing: Foundations and Applications} semester program, where preliminary discussions of this work took place. 
\addcontentsline{toc}{section}{References}
\bibliographystyle{plainurl}
\bibliography{ref}

\end{document}